\documentclass[11pt]{article}

\usepackage{amsmath}
\usepackage{amssymb}
\usepackage{amsthm}
\usepackage[margin = 1.2in]{geometry}
\usepackage{graphicx}
\usepackage[boxruled]{algorithm2e}
\usepackage{color}
\usepackage{cite}

\newcommand{\ip}[2]{\left\langle #1,  #2 \right\rangle}
\newcommand{\norm}[1]{\left\| #1 \right \|}
\newcommand{\tnorm}[1]{\left\| #1 \right \|}

\newcommand{\B}[2]{B \left( #1, #2 \right)}
\newcommand{\ls}[2]{\texttt{line\_search} \left(  #1, #2 \right)}
\newcommand{\argmin}{\operatornamewithlimits{argmin}}

\newcommand{\diag}[1]{\text{diag} \left(  #1 \right)}

\newcommand{\abs}[1]{|#1|}

\newcommand{\R}{\mathbb{R}}

\newcommand{\matrx}[1]{\begin{bmatrix} #1 \end{bmatrix} }

\newtheorem{thm}{Theorem}[section]
\theoremstyle{plain}

\newtheorem{lem}[thm]{Lemma}
\newtheorem{prop}[thm]{Proposition}
\newtheorem{rem}[thm]{Remark}

\title{An optimal first order method based on optimal
quadratic averaging}
\author{Dmitriy Drusvyatskiy\thanks{
University of Washington, Department of Mathematics, 
		Seattle, WA 98195; \texttt{ddrusv@uw.edu}. Research of Drusvyatskiy was partially supported by the AFOSR YIP award FA9550-15-1-0237.}
	\and Maryam Fazel\thanks{University of Washington, Department of Electrical Engineering, 
		Seattle, WA 98195; \texttt{mfazel@uw.edu}. Research partially supported by ONR award N00014-12-1-1002 and NSF award CIF-1409836.}
	\and 
	Scott Roy\thanks{University of Washington, Department of Mathematics, 
		Seattle, WA 98195; \texttt{scottroy@uw.edu}}
	}

\begin{document}
	\date{}
	\maketitle


	\begin{abstract}
		{In a recent paper, Bubeck, Lee, and Singh introduced a  new first order method for minimizing smooth strongly convex functions. Their geometric descent algorithm, largely inspired by the  
		ellipsoid method, enjoys the optimal linear rate of convergence. We show that the same iterate sequence is generated by a scheme that in each iteration computes an optimal average of quadratic lower-models of the function. Indeed, the minimum of the averaged quadratic approaches the true minimum at an optimal rate. 
		This intuitive viewpoint reveals clear connections to the original fast-gradient methods and cutting plane ideas, and leads to limited-memory extensions with improved  performance.}
	\end{abstract}

	\section{Introduction}
Consider a function $f\colon\R^n\to\R$ that is $\beta$-smooth and $\alpha$-strongly convex. Thus each point $x$ yields a quadratic upper estimator and a quadratic lower estimator of the function. Namely, 
inequalities $q(y; x)\leq f(y) \leq Q(y;x)$ 
hold for all $x,y\in \R^n$, where we set
\begin{align*}
q(y; x) &:= f(x) + \ip{ \nabla f(x) }{y - x} + \frac{\alpha}{2} \tnorm{y-x}^2,\\ 
Q(y;x) &:= f(x) + \ip{ \nabla f(x) }{y - x} + \frac{\beta}{2} \tnorm{y-x}^2.
\end{align*}
Classically, one step of the steepest descent algorithm decreases the squared distance of the iterate to the minimizer of $f$ by the fraction $1-\alpha/\beta$. This linear convergence rate is suboptimal from a computational complexity viewpoint. Optimal first-order methods, originating in Nesterov's work \cite{nest_orig} achieve the superior (and the best possible) linear rate $1-\sqrt{\alpha/\beta}$; see also the discussion in \cite[Section~2.2]{nestBook}. Such accelerated schemes, on the other hand, are notoriously difficult to analyze. Numerous recent papers (e.g. \cite{bubLeeSin,quad_const_acc,diff_accel,attouch_pey,lin_couple}) have aimed to shed new light on optimal algorithms. 

This manuscript is motivated by the novel geometric descent algorithm of Bubeck, Lee, and Singh \cite{bubLeeSin}. Their scheme is highly geometric, sharing some aspects with the ellipsoid method, and it achieves the optimal linear rate of convergence.
Moreover, the geometric descent algorithm often has much better practical performance than accelerated gradient methods; see the discussion in \cite{bubLeeSin}.
Motivated by their work, in this paper we propose
 an intuitive method that maintains a quadratic lower model of the objective function, whose minimal value converges to  
the true minimum at an optimal linear rate. We will show that the two methods are indeed equivalent in the sense that they produce the same iterate sequence.
The quadratic averaging viewpoint, however, has important advantages. First, it immediately yields a  comparison with the original accelerated gradient method \cite{nestBook,nest_orig} and cutting plane techniques. Secondly, quadratic averaging motivates a simple strategy for significantly accelerating the method in practice by utilizing accumulated information -- a limited memory version of the scheme.

The outline of the paper is as follows. In Section~\ref{sec:optAvg}, we describe the optimal quadratic averaging framework (Algorithm~\ref{alg:optQuadAvg}) -- the focal point of the manuscript.
In Section~\ref{sec:opt_quad_mem}, we propose a limited memory version of Algorithm~\ref{alg:optQuadAvg}, based on iteratively solving small dimensional quadratic programs.
In Section~\ref{sec:connect_geodesc}, we show that our Algorithm~\ref{alg:optQuadAvg} and the geometric descent method of \cite{bubLeeSin} produce the same iterate sequence.
Section~\ref{sec:numerical_examples} is devoted to numerical illustrations, in particular showing that the optimal quadratic averaging algorithm with memory can be competitive with L-BFGS.
We finish the paper with Section~\ref{sec:prox_ext}, where we discuss the challenges that must be overcome in order to derive proximal extensions. In the final stages of revising this paper, a new manuscript \cite{comp_geo_desc} appeared explaining how to 
overcome exactly these challenges.


\subsection{Notation}
We follow the notation of  \cite{bubLeeSin}. Given a point $x \in \R^n$, we define a \emph{short step}
\[ x^{+} := x - \frac{1}{\beta} \nabla f(x) \]
and a \emph{long step} 
\[ x^{++} := x - \frac{1}{\alpha} \nabla f(x). \]
Setting $y = x^+$  in the quadratic bound $f(y) \leq Q(y; x)$  yields the standard inequality 
\begin{equation}\label{eqn:desc}
f(x^+) + \frac{1}{2\beta} \tnorm{ \nabla f(x) }^2\leq f(x).
\end{equation}
We denote the unique minimizer of $f$ by $x^*$, its minimal value by $f^*$, and its condition number by $\kappa := \beta/\alpha$. Throughout, the symbol 
$B(x,R^2)$ stands for the Euclidean ball of radius $R$ around $x$. For any points $x,y\in \R^n$, we let $\ls{x}{y}$ be the minimizer of $f$ on the line between $x$ and $y$.

\section{Optimal quadratic averaging} \label{sec:optAvg}
The starting point for our development is the elementary observation that every point $\bar x$ provides a quadratic under-estimator of the objective function, having a canonical form. Indeed, completing the square in the strong convexity inequality $f(x) \geq q(x; \bar x)$ yields
\begin{equation} \label{eqn:lowerBound}
f(x) \geq \left(f(\bar x) - \frac{\tnorm{\nabla f(\bar x)}^2}{2 \alpha} \right)+ \frac{\alpha}{2} \tnorm{x - {\bar x}^{++}}^2.
\end{equation}
Suppose we have now available two quadratic lower-estimators:
\[ f(x) \geq Q_A(x) := v_A + \frac{\alpha}{2} \tnorm{ x - x_A }^2 \qquad \text{and} \qquad f(x) \geq Q_B(x) := v_B + \frac{\alpha}{2} \tnorm{ x - x_B }^2. \] 
Clearly, the minimal values of $Q_A$ and of $Q_B$ lower-bound the minimal value of $f$.
For  any $\lambda \in [0,1]$, the average $Q_{\lambda}:=\lambda Q_A+(1-\lambda)Q_B$  is again a quadratic lower-estimator of $f$. Thus we are led to the question:
\begin{center}
What choice of $\lambda$ yields the tightest lower-bound on the minimal value of $f$?
\end{center}
To answer this question, observe the equality
\begin{align*}
Q_{\lambda}(x):=\lambda Q_A(x) + (1 - \lambda) Q_B(x) 
= v_{\lambda}+ \frac{\alpha}{2} \tnorm{x - c_{\lambda}}^2,
\end{align*}
where 
$$c_{\lambda}=\lambda x_A + (1-\lambda) x_B$$
and 
\begin{equation}\label{eqn:v_lam}
v_{\lambda}= v_B + \left( v_A - v_B + \frac{\alpha}{2} \tnorm{x_A - x_B}^2 \right) \lambda - \left(\frac{\alpha}{2} \tnorm{x_A-x_B}^2 \right) \lambda^2.
\end{equation}
In particular, the average $Q_\lambda$ has the same canonical form as $Q_A$ and $Q_B$. 
A quick computation now shows that $v_{\lambda}$ (the minimum of $Q_\lambda$) is maximized  by setting
\[ \bar\lambda := \text{proj}_{[0,1]} \left( \frac{1}{2} + \frac{v_A - v_B}{\alpha \tnorm{x_A-x_B}^2} \right). \]
With this choice of $\lambda$, we call the quadratic function $\overline Q=\bar v+\frac{\alpha}{2}\|\cdot-\bar c\|^2$ the {\em optimal averaging} of 
$Q_A$ and $Q_B$. See Figure~\ref{fig:optAvg} for an illustration.
\begin{figure}[h!]
\centering
\includegraphics[scale=0.5]{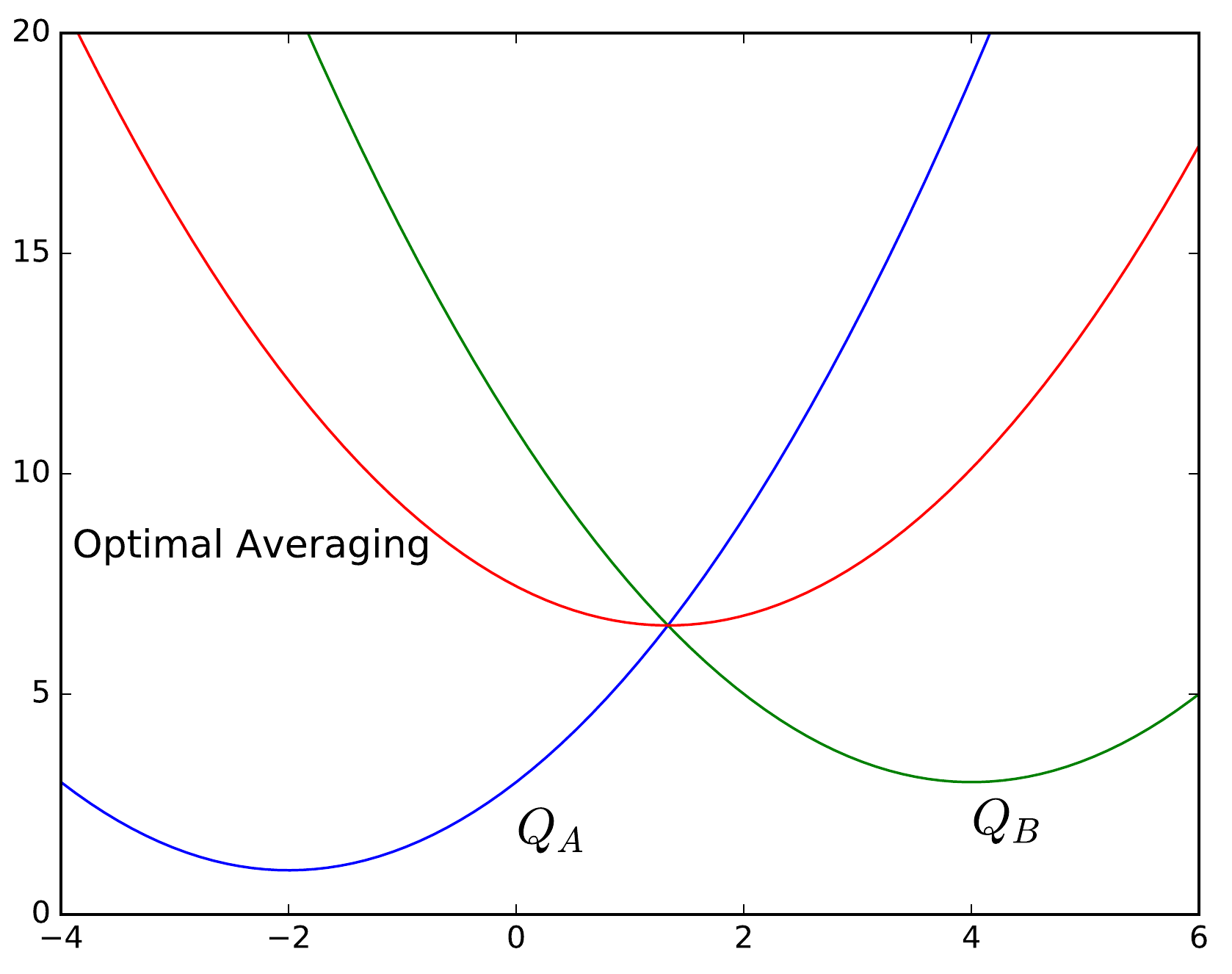}
\caption{The optimal averaging of $Q_A(x) = 1 + 0.5 ( x+2)^2$ and $Q_B(x) = 3 + 0.5 ( x - 4)^2$. }  \label{fig:optAvg}
\end{figure}

An algorithmic idea emerges. Given a current iterate $x_k$, form the quadratic lower-model $Q(\cdot)$ in \eqref{eqn:lowerBound} 
with $\bar x=x_k$. Then let $Q_{k}$ be the optimal averaging of $Q$ and the quadratic lower model $Q_{k-1}$ from the previous step. Finally define $x_{k+1}$ to be the minimizer of $Q_{k}$, and repeat. Though attractive, the scheme does not converge at an optimal rate. Indeed, this algorithm is closely related to the suboptimal method in \cite{bubLeeSin}; see Section~\ref{subsec:subopt_comp} for a discussion. The main idea behind acceleration, natural in retrospect, is a separation of roles: one must maintain two sequences of points $x_k$ and $c_k$. The points $x_k$  will generate quadratic lower models as above, while $c_k$ will be the minimizers of the quadratics.
We summarize the proposed method in Algorithm~\ref{alg:optQuadAvg}.
The rule for determining the iterate $x_{k}$ by a line search is entirely motivated by the geometric descent method in \cite{bubLeeSin}.

\begin{algorithm}[h!]
\caption{Optimal Quadratic Averaging}  \label{alg:optQuadAvg}
\KwIn{Starting point $x_0$ and strong convexity constant $\alpha > 0$.}
\KwOut{Final quadratic $Q_K(x) = v_K + \frac{\alpha}{2} \tnorm{x - c_K}^2$ and $x_K^+$.}
Set $Q_0(x) = v_0 + \frac{\alpha}{2} \tnorm{x - c_0}^2$, where $v_0 =  f(x_0 ) - \frac{\tnorm{\nabla f(x_0)}^2}{2 \alpha}$ and $c_0 = x_0^{++}$\;
\For{k = 1, \ldots, K}{ 
Set $x_{k} = \ls{c_{k-1}}{x_{k-1}^+}$\;
Set $Q(x) =  \left(f(x_{k}) - \frac{\tnorm{\nabla f(x_{k})}^2}{2 \alpha}\right) + \frac{\alpha}{2} \tnorm{ x - x_{k}^{++}}^2$ \;
Let $Q_{k}(x) = v_{k}+\frac{\alpha}{2}\|x-c_{k}\|^2$ be the optimal averaging of $Q$ and $Q_{k-1}$ \; }
\end{algorithm}

\begin{rem}
{\rm
When implementing Algorithm~\ref{alg:optQuadAvg}, we set $x_k^+=\ls{x_k}{x_ k - \nabla f(x_k)}$. This does not impact the analysis as $x_k^+$ still satisfies the key inequality~\eqref{eqn:desc}. With this modification, the algorithm does not require $\beta$ as part of the input, and we have observed that the algorithm performs better numerically.}
\end{rem}

To aid in the analysis of the scheme, we record the following easy observation.
\begin{lem} \label{lem:optAvg}
Suppose that $\overline Q=\bar v+\frac{\alpha}{2}\|\cdot-\bar c\|^2$ is the optimal averaging of the quadratics $Q_A=v_A+\frac{\alpha}{2}\|\cdot-x_A\|^2$ and $Q_B=v_B+\frac{\alpha}{2}\|\cdot-x_B\|^2$. Then the quantity $\bar v$ is nondecreasing in both $v_A$ and $v_B$. Moreover, whenever the inequality $|v_A-v_B|\leq \frac{\alpha}{2}\|x_A-x_B\|^2$ holds, we have
$$\bar v=\frac{\alpha}{8}\|x_A-x_B\|^2+\frac{1}{2}(v_A+v_B)+\frac{1}{2\alpha}\left(\frac{v_A-v_B}{\|x_A-x_B\|}\right)^2.$$
\end{lem}
\begin{proof}
Define $\hat{\lambda} := \frac{1}{2} + \frac{v_A - v_B}{\alpha \tnorm{x_A-x_B}^2}$.
Notice that we have
\begin{align*}
\hat{\lambda} \in [0,1] &\quad \text{if and only if} \quad |v_A-v_B|\leq \frac{\alpha}{2}\|x_A-x_B\|^2.
\end{align*}
If $\hat{\lambda}$ lies in $[0,1]$, equality $\bar \lambda = \hat \lambda$ holds, and then from \eqref{eqn:v_lam} we deduce
\[ \bar v = v_{\bar \lambda} = \frac{\alpha}{8}\|x_A-x_B\|^2+\frac{1}{2}(v_A+v_B)+\frac{1}{2\alpha}\left(\frac{v_A-v_B}{\|x_A-x_B\|}\right)^2. \]
If $\hat{\lambda}$ does not lie in $[0,1] $, then an easy argument shows that $\bar v$ is linear in $v_A$ either with slope one or zero. If $\hat{\lambda}$  lies in $(0,1)$, then we compute
\[ \frac{\partial \bar v}{\partial  v_A} = \frac{1}{2} + \frac{1}{\alpha \tnorm{x_A - x_B}^2} (v_A - v_B) ,\]
which 
is nonnegative because $\frac{\abs{v_A - v_B}}{\alpha \tnorm{x_A - x_B}^2} \leq \frac{1}{2}$. Since $\bar v$ is clearly continuous, it follows that $\bar v$ is nondecreasing in $v_A$, and by symmetry also in  $v_B$.
\end{proof}

We now show that Algorithm~\ref{alg:optQuadAvg} achieves the optimal linear rate of convergence.
\begin{thm}[Convergence of optimal quadratic averaging] \label{thm:optAvgConvergence}
In Algorithm~\ref{alg:optQuadAvg}, for every index $k\geq 0$, the inequalities $v_k\leq f^*\leq f(x_k^+)$ hold and we have
\[ f(x_k^+) - v_k \leq   \left( 1 - \frac{1}{\sqrt{\kappa}} \right)^k (f(x_0^+) - v_0).\]
\end{thm}

\begin{proof}
Since  in each iteration, the algorithm only averages quadratic minorants of $f$, the inequalities $v_k\leq f^*\leq f(x_k^+)$ hold for every index $k$.
Set $r_0=\frac{2}{\alpha} (f(x_0^+) - v_0)$ and define the quantities $r_k := \left( 1 - \frac{1}{\sqrt{\kappa}} \right)^k r_0$.
We will show by induction that the inequality $v_k \geq  f(x_k^+) - \frac{\alpha}{2} r_k$ holds for  all $k \geq 0$.
The base case $k=0$ is immediate, and so assume we have
\[ v_{k-1} \geq f(x_{k-1}^+) - \frac{\alpha}{2} r_{k-1} \]
for some index $k-1$.
Next set $v_A := f(x_{k}) - \frac{\tnorm{\nabla f(x_{k})}^2}{2 \alpha}$ and $v_B := v_{k-1}$.
Then the function
\[ Q_{k}(x) = v_{k} + \frac{\alpha}{2} \tnorm{x - c_{k}}^2, \]
is the optimal averaging of $Q_A(x) =v_A + \frac{\alpha}{2} \tnorm{ x - x_{k}^{++} }^2$ and $Q_B(x) = v_B + \frac{\alpha}{2} \tnorm{x - c_{k-1}}^2$.
An application of \eqref{eqn:desc} yields the lower bound $\hat{v}_A$ on $v_{A}$:
\begin{align*}
v_A &= f(x_{k}) - \frac{\tnorm{\nabla f(x_{k})}^2}{2 \alpha}  
\geq f(x_{k}^+) - \frac{\alpha}{2} \frac{\tnorm{\nabla f(x_{k})}^2}{\alpha^2} \left( 1 - \frac{1}{\kappa} \right) := \hat{v}_A.
\end{align*}
The induction hypothesis and the choice of $x_{k}$ yield a lower bound $\hat{v}_B$ on $v_B$:
\begin{align*}
v_B \geq f(x_{k-1}^+) - \frac{\alpha}{2} r_{k-1} &\geq f(x_{k}) - \frac{\alpha}{2} r_{k-1} \\
&\geq f(x_{k}^+) + \frac{1}{2\beta} \tnorm{ \nabla f(x_{k}) }^2 - \frac{\alpha}{2} r_{k-1} \\
&= f(x_{k}^+) - \frac{\alpha}{2} \left( r_{k-1} - \frac{1}{\alpha^2 \kappa} \tnorm{ \nabla f(x_{k}) }^2 \right) := \hat{v}_B.
\end{align*}

Define the quantities $d := \tnorm{x_{k}^{++} - c_{k-1}}$ and  $h := \frac{\tnorm{\nabla f(x_{k})}}{\alpha}$.
We now split the proof into two cases.
First assume $h^2\leq \frac{r_{k-1}}{2}$.
Then we deduce
\begin{align*}
v_{k} \geq v_{A}\geq \hat{v}_A&= f(x_{k}^+) - \frac{\alpha}{2} h^2 \left( 1 - \frac{1}{\kappa} \right) \\
&\geq  f(x_{k}^+) - \frac{\alpha}{2} r_{k-1} \left( \frac{1- \frac{1}{\kappa}}{2} \right) \\
&\geq f(x_{k}^+) - \frac{\alpha}{2} r_{k-1} \left( 1 - \frac{1}{\sqrt{\kappa}} \right) \\
&= f(x_{k}^+) - \frac{\alpha}{2} r_{k},
\end{align*}
where the third line follows since $2/\sqrt{\kappa}\leq 1+1/\kappa$ holds.
Hence in this case, the proof is complete.

Next suppose $h^2 > \frac{r_{k-1}}{2}$ and let $v+\frac{\alpha}{2}\|\cdot-c\|^2$ be the optimal average of the two quadratics $\hat{v}_A+\frac{\alpha}{2}\|\cdot - x_{k}^{++}\|^2$ and $\hat{v}_B+\frac{\alpha}{2}\|\cdot-c_{k-1}\|^2$. By Lemma~\ref{lem:optAvg}, the inequality $v_{k}\geq v$ holds.
We claim that equality 
\begin{equation}\label{eqn:represent_v}
v = \hat{v}_B + \frac{\alpha}{8} \ \frac{ ( d^2 + \frac{2}{\alpha}( \hat{v}_A - \hat{v}_B ) )^2 }{d^2}\qquad \textrm{ holds}.
\end{equation}
This follows immediately from Lemma~\ref{lem:optAvg}, once we show $\frac{1}{2} \geq \frac{\abs{\hat{v}_A-\hat{v}_B}}{\alpha d^2}$. To this end, note first the equality $\frac{\abs{\hat{v}_A-\hat{v}_B}}{\alpha d^2}=\frac{\abs{r_{k-1} - h^2}}{2d^2}$.
The choice  $x_{k}=\ls{c_{k-1}}{x_{k-1}^+}$ ensures: 
$$d^2-h^2= \|x_k-c_{k-1}\|^2-\frac{2}{\alpha} \langle \nabla f(x_k),x_k-c_{k-1} \rangle=\|x_k-c_{k-1}\|^2\geq 0 .$$
Thus we have $h^2 - r_{k-1} < h^2 \leq d^2$.
Finally, the assumption $h^2 > \frac{r_{k-1}}{2}$ implies
\begin{equation}\label{eqn:little}
r_{k-1} - h^2 < \frac{r_{k-1}}{2} < h^2 \leq d^2.
\end{equation}
Hence we can be sure that \eqref{eqn:represent_v} holds.
Plugging in $\hat{v}_A$ and $\hat{v}_B$ yields
\begin{align*}
v &=  f(x_{k}^+) - \frac{\alpha}{2} \left( r_{k-1} - \frac{1}{\kappa} h^2 - \frac{ ( d^2 + r_{k-1} - h^2)^2 }{4 d^2} \right).
\end{align*}
Hence the proof is complete once we show the inequality
\[ r_{k-1} - \frac{1}{\kappa} h^2 - \frac{(d^2 + r_{k-1} - h^2)^2}{4 d^2} \leq \left( 1 - \frac{1}{\sqrt{\kappa}} \right) r_{k-1}.  \]
After rearranging, our task simplifies to showing the inequality
\[ \frac{r_{k-1}}{\sqrt{\kappa}} \leq \frac{h^2}{\kappa} + \frac{(d^2 + r_{k-1} - h^2)^2}{4 d^2}. \]
Taking derivatives and using inequality \eqref{eqn:little}, one can readily verify that the right-hand-side is nondecreasing in $d^2$ on the interval  $d^2\in [h^2,+\infty)$.
Thus plugging in the endpoint $d^2=h^2$ we deduce 
\[ \frac{h^2}{\kappa} + \frac{(d^2 + r_{k-1} - h^2)^2}{4 d^2} \geq \frac{h^2}{\kappa}  + \frac{r_{k-1}^2}{4 h^2}. \]
Minimizing the right-hand-side over all $h$ satisfying $h^2 \geq \frac{r_{k-1}}{2}$ yields the  inequality
\[ \frac{h^2}{\kappa} + \frac{r_{k-1}^2}{4 h^2}  \geq \frac{r_{k-1}}{\sqrt{\kappa}}. 
\]
The proof is complete.
\end{proof}


%

It is instructive to compare optimal averaging (Algorithm~\ref{alg:optQuadAvg}) with Nesterov's optimal methods in \cite{nestBook,nest_orig}. 
For convenience, we record the optimal gradient method following \cite{nestBook}, in Algorithm~\ref{alg:nestOpt}.

\begin{algorithm}[h!]
	\caption{General scheme of an optimal method [Nesterov]}  \label{alg:nestOpt}
	\KwIn{Starting points $x_0$ and $c_0$, strong convexity constant $\alpha > 0$, smoothness parameter $\beta > 0$, and initial quadratic curvature $\gamma_0 \geq \alpha$.}
	\KwOut{Final quadratic $Q_K(x) = v_K + \frac{\gamma_K}{2} \tnorm{x - c_K}^2$.}
	Set $Q_0(x) = v_0 + \frac{\gamma_0}{2} \tnorm{x - c_0}^2$, where $v_0 = f(x_0) - \frac{1}{2\beta} \tnorm{\nabla f(x_0)}^2$ \;
	\For{k = 1, \ldots, K}{ 
		Compute averaging parameter $\lambda_{k} \in (0,1)$ from $\beta \lambda_{k}^2 = (1-\lambda_{k}) \gamma_{k-1} + \lambda_{k} \alpha$ \;
		Set $\gamma_{k} = (1-\lambda_{k}) \gamma_{k-1} + \lambda_{k} \alpha$. \;
		Set $x_{k} = (1 - \theta_{k}) c_{k-1} + \theta_{k} x_{k-1}^+$ where $\theta_{k} = \frac{\gamma_{k}}{\gamma_{k-1} + \lambda_{k} \alpha}$   \;
		Set $Q(x) =  \left(f(x_{k}) - \frac{\tnorm{\nabla f(x_{k})}^2}{2 \alpha}\right) + \frac{\alpha}{2} \tnorm{ x - x_{k}^{++}}^2$ \;
		Let $c_{k}$ be the minimizer of the quadratic $Q_{k}(x) = (1-\lambda_{k}) Q_{k-1}(x) + \lambda_{k} Q(x)$ \;} 
	\tcc{If we set $\gamma_0 = \alpha$, then we have $\gamma_k = \alpha$, $\lambda_k = \frac{1}{\sqrt{\kappa}}$, and $\theta_k = \frac{\sqrt{\kappa}}{1 + \sqrt{\kappa}}$.}
\end{algorithm}

 Comparing Algorithms~\ref{alg:optQuadAvg} and \ref{alg:nestOpt}, we see that
\begin{itemize}
	\item $x_{k}$ is some point on the line between $c_{k-1}$ and $x_{k-1}^+$, and
	\item $Q_{k}$ is an average of the previous quadratic $Q_{k-1}$ and the strong convexity quadratic lower bound $Q$ based at $x_{k}$.
\end{itemize}
As we discuss in Appendix~\ref{sec:exact_line_search}, we can modify Nesterov's method so that like in optimal quadratic averaging, we set $x_{k} = \ls{c_{k-1}}{x_{k-1}^+}$ in each iteration.
After this change, only two differences remain between the schemes:
\begin{itemize}
	\item the initial quadratic $Q_0$ is different, and
	\item the averaging parameter is computed differently.
\end{itemize}
These differences, however, are fundamental.
In Algorithm~\ref{alg:optQuadAvg}, the quadratic $Q_0$  lower bounds $f$ and therefore optimal averaging makes sense; in the accelerated gradient method, $Q_0$ does not lower bound $f$, and the idea of optimal averaging does not apply.

\section{Optimal quadratic averaging with memory}\label{sec:opt_quad_mem}
Each iteration of Algorithm~\ref{alg:optQuadAvg} forms an optimal average of the current lower quadratic model with the one from the previous iteration; that is, as stated the scheme has a memory size of one. We next show how the scheme easily adapts to  maintaining limited memory, i.e. by averaging multiple quadratics in each iteration. We mention in passing that the authors of \cite{bubLeeSin} left open the question of efficiently speeding up their geometric descent algorithm in practice. One approach of this flavor has recently appeared in \cite[Section 4]{bubLeenew}. The optimal averaging viewpoint, developed here, provides a direct and satisfying alternative.
Indeed, computing the optimal average of several quadratics is easy, and amounts to solving a small dimensional quadratic optimization problem.

To see this, fix $t$ quadratics $Q_i(x) := v_i + \frac{\alpha}{2} \tnorm{x - c_i}^2$, with $i \in \{1, \ldots, t \}$, and a weight vector $\lambda$ in the $t$-dimensional simplex $\Delta_t := \left\{ x \in \R^t : \sum_{i=1}^t x_i = 1, \ x \geq 0 \right\}$.
The average quadratic
\[ Q_{\lambda}(x) := \sum_{i=1}^t \lambda_i Q_i(x) \]
maintains the same canonical form as each $Q_i$.
\begin{prop} \label{prop:optavg}
Define the matrix $C = \matrx{ c_1 & c_2 & \ldots & c_t }$ and  vector $v = \matrx{ v_1 & v_2 & \ldots & v_t }^T$.
Then we have
\[ Q_{\lambda}(x) = v_{\lambda} + \frac{\alpha}{2} \tnorm{ x - c_{\lambda} }^2, \]
where
\[ c_{\lambda} = C \lambda \quad \text{and} \quad v_{\lambda} = \ip{ \frac{\alpha}{2} {\rm diag\,}{(C^T C) } + v}{ \lambda} - \frac{\alpha}{2} \tnorm{C \lambda}^2. \]
\end{prop}
\begin{proof}
The Hessian of $Q_{\lambda}$ is simply $\frac{\alpha}{2}I$, and 
therefore the quadratic $Q_{\lambda}(x)$ has the form
\[ v_{\lambda} + \frac{\alpha}{2} \tnorm{ x - c_{\lambda} }^2\]
for some $v_{\lambda}$ and $c_{\lambda}$.
Notice that $c_{\lambda}$ is the minimizer of $Q_{\lambda}$, and by differentiating, we determine that $c_{\lambda} = \sum_{i=1}^t \lambda_i c_i = C \lambda$.
We then compute
\begin{align*}
v_{\lambda} = Q_{\lambda} ( c_{\lambda} ) &= \sum_{i=1}^t \left( \lambda_i v_i + \frac{\lambda_i \alpha}{2} \tnorm{C \lambda - c_i}^2 \right) \\
&= \ip{v}{\lambda} + \frac{\alpha}{2} \sum_{i=1}^t \lambda_i \left( \tnorm{C \lambda}^2 - 2 \ip{C \lambda}{c_i} + \tnorm{c_i}^2 \right) \\
&= \ip{v}{\lambda} + \frac{\alpha}{2} \tnorm{C \lambda}^2 - \alpha \ip{C \lambda}{ \sum_{i=1}^t \lambda_i c_i } + \frac{\alpha}{2} \sum_{i=1}^t \lambda_i \tnorm{c_i}^2\\
&= \ip{ \frac{\alpha}{2} \diag{ C^T C } + v}{ \lambda} - \frac{\alpha}{2} \tnorm{C \lambda}^2.
\end{align*}
The proof is complete.
\end{proof}

Naturally, we define the \emph{optimal averaging} of the quadratics $Q_i$, with $i \in \{1, 2, \ldots, t\}$, to be $Q_{\bar \lambda}$, where $\bar \lambda$ is the maximizer of the concave quadratic over the simplex:
\[ \min_{\lambda\in \Delta_t}~ v_{\lambda} = \ip{ \frac{\alpha}{2} \diag{ C^T C } + v}{ \lambda} - \frac{\alpha}{2} \tnorm{C \lambda}^2.\]
There is no closed form expression for $\bar \lambda$, but one can quickly find it by solving a quadratic program in $t$ variables, for example by an active set method. Moreover, some thought shows that the matrix $C^TC$ can be efficiently updated if one of the centers changes; we omit the details.

We propose an optimal averaging scheme with memory in Algorithm~\ref{alg:optAvgWithMem}.
As we see in Section~\ref{sec:numerical_examples}, the method performs well numerically.
Moreover, the scheme enjoys the same convergence guarantees as Algorithm~\ref{alg:optQuadAvg}; that is, Theorem~\ref{thm:optAvgConvergence} applies to Algorithm~\ref{alg:optAvgWithMem}, with nearly the same proof (which we omit).




\begin{algorithm}
\caption{Optimal Quadratic Averaging with Memory}  \label{alg:optAvgWithMem}
\KwIn{Starting point $x_0$, strong convexity constant $\alpha > 0$, and memory size $t \geq 1$.}
\KwOut{Final quadratic $Q_K(x) = v_K + \frac{\alpha}{2} \tnorm{x - c_K}^2$ and $x_K^+$.}
Set $Q_0(x) = v_0 + \frac{\alpha}{2} \tnorm{x - c_0}^2$, where $v_0 =  f(x_0 ) - \frac{\tnorm{\nabla f(x_0)}^2}{2 \alpha}$ and $c_0 = x_0^{++}$ \;
\For{k = 1, \ldots, K}{
Set $x_{k} = \ls{c_{k-1}}{x_{k-1}^+}$\;
Set $M_{k}(x) = f(x_{k}) - \frac{\tnorm{\nabla f(x_{k})}^2}{2 \alpha}  + \frac{\alpha}{2} \tnorm{x - x_{k}^{++} }^2$ \;
Let $Q_k(x) := v_k + \frac{\alpha}{2} \tnorm{x - c_k}^2$ be the optimal averaging of the
\begin{align*}
k + 1 \text{ quadratics } &\quad Q_{k-1}, \ M_k, \ M_{k-1}, \ \ldots, \ M_1 & \quad \text{if } & k \leq t, \text{ or of the} \\
t + 1 \text{ quadratics } &\quad Q_{k-1}, \ M_k, \ M_{k-1}, \ \ldots, \ M_{k-t +1} &\quad \text{if } & k \geq t + 1 ;
\end{align*} }
\end{algorithm}

The reader may notice that Algorithm~\ref{alg:optAvgWithMem} shows some similarity to the classical Kelley's method for minimizing nonsmooth convex functions \cite{kel}. In the simplest case of minimizing a smooth convex function $f$ on $\R^n$, Kelley's method iterates the following steps
$$x_{k+1}=\argmin_{x}\, f_k(x)$$
for the functions
$$f_k(x):=\max_{i=1,\ldots,k} \{f(x_i)+\langle \nabla f(x_i),x-x_i\rangle\}.$$
In other words, the scheme iteratively minimizes the (piecewise linear) lower-models $f_k$ of $f$. Coming back to the optimal averaging viewpoint, suppose that $Q_{\bar \lambda}$ is an optimal average of the lower-bounding quadratics $Q_i$, for $i=1,\ldots,k$. Then we may write
\begin{align*}
v_{\bar{\lambda}}&=\max_{\lambda\in \Delta_k}\, \min_x\, \sum_{i}\lambda_i Q_i(x)
=\min_x\, \max_{\lambda\in \Delta_k}\,\sum_{i}\lambda_i Q_i(x)
=\min_x\, \left(\max_{i=1,\ldots,k}\, Q_i(x)\right)
\end{align*}
Thus $v_{\bar{\lambda}}$ is the minimal value of the now different lower-model, $\max_{i=1,\ldots,k}\, Q_i$, of $f$. Kelley's method is known to have poor numerical performance and convergence guarantees (e.g. \cite[Section 3.3.2]{nestBook}), while Algorithm~\ref{alg:optAvgWithMem} achieves the optimal linear convergence rate. This disparity is of course based on the two key distinctions: (1) using quadratic lower-models coming from strong convexity instead of linear functions, and (2) maintaining two separate sequences $c_k$ (centers) and $x_k$ (sources of lower model updates).

\section{Equivalence to geometric descent}\label{sec:connect_geodesc}
Algorithm~\ref{alg:optQuadAvg} is largely motivated by the geometric descent method introduced by Bubeck, Lee, and Singh \cite{bubLeeSin}.
In this section, we show the two methods (Algorithm 1 and Algorithm 4) indeed generate an identical iterate sequence.

\subsection{Suboptimal geometric descent method}\label{subsec:subopt_comp}
The basic idea of geometric descent \cite{bubLeeSin} is
that for each point $x \in \R^n$, the strong convexity lower bound $f^* \geq q(x^*; x)$ defines a ball containing $x^*$:
\[ x^* \in \B{x^{++}}{ \frac{\tnorm{\nabla f(x)}^2}{\alpha^2} - \frac{2}{\alpha} \left( f(x) - f^* \right)  }. \] 
In turn, taking into account \eqref{eqn:desc} yields the guarantee
\begin{equation}\label{eqn:main_inc}
 x^* \in \B{x^{++}}{ \left(1 - \frac{1}{\kappa} \right) \frac{\tnorm{\nabla f(x)}^2}{\alpha^2} - \frac{2}{\alpha} \left( f(x^+) - f^* \right)  }. 
 \end{equation}
A crude upper estimate of the radius above is obtained simply by ignoring the nonnegative term $\frac{2}{\alpha} \left( f(x^+) - f^* \right)$.
The suboptimal geometric descent method proceeds as follows.
Suppose we have available some ball $\B{c_0}{R_0^2}$ containing $x^*$.
As discussed, the quadratic lower bound at the center $c_0$, namely $f^*\geq q(x^*,c_0)$, yields another ball $\B{c_0^{++}}{ \left(1 - \frac{1}{\kappa} \right) \frac{\tnorm{\nabla f(c_0)}^2}{\alpha^2}}$ containing $x^*$.
Geometrically it is clear that the intersection of these two balls must be significantly smaller than either of the individual balls. 
The following lemma from \cite{bubLeeSin} makes this observation precise; see Figure~\ref{fig:illust_lemma} for an illustration.
\begin{lem} [Minimal enclosing ball of the intersection] \label{lem:ballIntersection1}
	Fix a center $x \in \R^n$, square radius $R^2 > 0$, step $h \in \R^n$, and $\epsilon \in (0,1)$.
	Then there exists a new center $c \in \R^n$ with
	\[  \B{x}{R^2} \cap \B{x + h}{(1- \epsilon) \tnorm{h}^2} \subset \B{c}{(1 - \epsilon) R^2}. \]
\end{lem}
\smallskip
\noindent
An application of Lemma~\ref{lem:ballIntersection1} yields a
new center $c_1$ with
\[ \B{c_0}{R_0^2} \cap \B{c_0^{++}}{ \left(1 - \frac{1}{\kappa} \right) \frac{\tnorm{\nabla f(c_0)}^2}{\alpha^2}} \subset \B{c_1}{ \left(1 - \frac{1}{\kappa} \right) R_0^2}. \]
Repeating the procedure with the new ball $\B{c_1}{\left(1 - \frac{1}{\kappa} \right) R_0^2}$ yields a sequence of centers $c_k$ satisfying
\[ \tnorm{c_k - x^*}^2 \leq \left( 1 - \frac{1}{\kappa} \right)^k R_0^2. \]
We note that the centers $c_k$ and $R_0^2$ of the minimal enclosing balls in Lemma~\ref{lem:ballIntersection1} are easy to compute; see Algorithm~1 in \cite{bubLeeSin}.

\begin{figure}[!ht]
	\centering
	\includegraphics[width=0.4\textwidth]{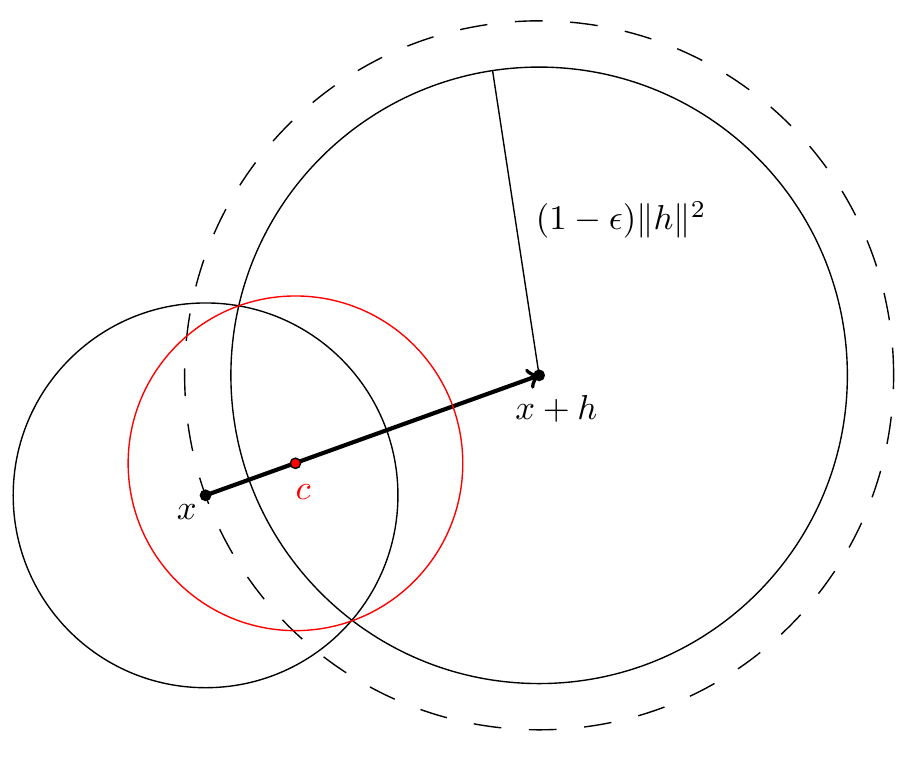}
	\caption{Minimal enclosing ball of the intersection.}
	\label{fig:illust_lemma}
\end{figure}

There is a very close connection between finding the minimal enclosing ball of the intersection of two balls and of optimally averaging quadratics. To see this, consider again two quadratics 
\[ f(x) \geq Q_A(x) := v_A + \frac{\alpha}{2} \tnorm{ x - x_A }^2 \qquad \text{and} \qquad f(x) \geq Q_B(x) := v_B + \frac{\alpha}{2} \tnorm{ x - x_B }^2. \] 
Let $\overline{Q}$ be the optimal average of $Q_A$ and $Q_B$.
Notice that since $Q_A$, $Q_B$, and $\overline Q$  lower bound $f$, the minimizer $x^*$ of $f$ is guaranteed to lie in the three balls:
\begin{align*}
\B{x_A}{R_A^2} &\quad \text{where} \quad R_A^2 = \frac{2}{\alpha} \left( \hat f - v_A \right), \\
\B{x_B}{R_B^2} &\quad \text{where} \quad R_B^2 = \frac{2}{\alpha} \left( \hat f - v_B \right),  \\
\B{\bar c}{ R^2} &\quad \text{where} \quad R^2 = \frac{2}{\alpha} \left( \hat f - \bar v\right),
\end{align*}
where $\hat f$ is any upper bound on $f^*$.
We observe the following elementary fact. 

\begin{prop}[Minimal enclosing ball and optimal averaging]\label{prop:rel_min_enc_opt_ave} 
The ball $\B{\bar c}{R^2}$ is precisely the minimal enclosing ball of the intersection $\B{x_A}{R_A^2} \cap \B{x_B}{R_B^2}$.
\end{prop}
\begin{proof}
Define the quantity $\hat{\lambda} = \frac{1}{2} + \frac{v_A - v_B}{\alpha \tnorm{x_A-x_B}^2}$.
If $\hat{\lambda}$ lies in the unit interval $[0,1]$, then a quick computation using Lemma~\ref{lem:optAvg} shows the expressions
\begin{align*}
R^2 &= R_B^2 -  \frac{\left( \tnorm{x_A - x_B}^2 + R_B^2 - R_A^2 \right)^2}{4 \tnorm{x_A-x_B}^2}
\end{align*}
and
\begin{align*}
\bar c &= \bar \lambda x_A + (1- \bar \lambda ) x_B = \frac{1}{2} \left( x_A + x_B \right) - \frac{ R_A^2 - R_B^2 }{2 \tnorm{x_A - x_B}^2 } \left( x_A - x_B \right).
\end{align*}
Now observe
\begin{align*}
\hat{\lambda} < 0 &\quad \text{if and only if} \quad \tnorm{x_A-x_B}^2 < R_A^2 - R_B^2 \\
\hat{\lambda} \in [0,1] &\quad \text{if and only if} \quad \tnorm{x_A-x_B}^2 \geq \abs{R_A^2 - R_B^2}, \text{ and} \\
\hat{\lambda} > 1 &\quad \text{if and only if} \quad  \tnorm{x_A-x_B}^2 < R_B^2 - R_A^2.
\end{align*}
Comparing with the recipe \cite[Algorithm~1]{bubLeeSin} for computing the minimal enclosing ball, we see that $\B{\bar c}{R^2}$ is the minimal enclosing ball of the intersection $\B{x_A}{R_A^2} \cap \B{x_B}{R_B^2}$. 
\end{proof}

\subsection{Optimal geometric descent method}\label{subsec:opt_geo}
To obtain an optimal method, the authors of \cite{bubLeeSin} observe that the term $\frac{2}{\alpha} \left( f(x^+) - f^* \right)$ in the inclusion \eqref{eqn:main_inc} cannot be ignored. 
Exploiting this term will require maintaining two sequences $c_k$ (the centers of the balls) and $x_k$ (points for generating new balls).
Suppose in iteration $k$, we know that $x^*$ lies in the ball
\[ \B{c_k}{R_k^2 - \frac{2}{\alpha} \left( f(x_k^+) - f^* \right) }.\]
Consider now an arbitrary point, denoted suggestively by $x_{k+1}$. Then  
\eqref{eqn:main_inc} implies the inclusion
\begin{equation}\label{eqn:ball2}
 x^* \in \B{x_{k+1}^{++}}{ \left(1 - \frac{1}{\kappa} \right) \frac{\tnorm{\nabla f(x_{k+1})}^2}{\alpha^2} - \frac{2}{\alpha} \left( f(x_{k+1}^+) - f^* \right)  }. 
\end{equation}
If we choose $x_{k+1}$ to satisfy $f(x_{k+1}) \leq f(x_k^+)$ and apply inequality \eqref{eqn:desc} with $x=x_{k+1}$, we can get a new upper estimate of the initial ball,
\begin{equation}\label{eqn:ball1}
 x^* \in \B{c_k}{R_k^2 - \frac{1}{\kappa}~\frac{ \tnorm{\nabla f(x_{k+1})}^2 }{\alpha^2} - \frac{2}{\alpha} \left( f(x_{k+1}^+) - f^* \right)}.
\end{equation}
It seems clear that if the centers $c_k$ and $x_{k+1}^{++}$ of the two balls in \eqref{eqn:ball2} and \eqref{eqn:ball1} are ``sufficiently far apart'', then their intersection is contained in an even smaller ball. 
This is the content of following lemma from \cite{bubLeeSin}.
\begin{lem}[Two balls shrinking] \label{lem:ballIntersection2}
	Fix centers $x_A, x_B \in \R^n$ and square radii $r_A^2, r_B^2 > 0$.
	Also fix $\epsilon \in (0,1)$ and suppose $\tnorm{x_A - x_B}^2 \geq r_B^2$.
	Then there exists a new center $c \in \R^n$ such that for any $\delta > 0$, we have
	\[  \B{x_A}{r_A^2 - \epsilon r_B^2 - \delta} \cap \B{x_B}{(1- \epsilon) r_B^2 - \delta} \subset \B{c}{(1 - \sqrt{\epsilon}) r_A^2  - \delta}. \]
\end{lem}
 A quick application of this result shows that provided 
 \begin{equation}\label{eqn:separat}
  \tnorm{x_{k+1}^{++} - c_k}^2 \geq \frac{\tnorm{\nabla f(x_{k+1})}^2}{\alpha^2}
  \end{equation}
holds, there exists a new center $c_{k+1}$ with
\[ x^* \in \B{c_{k+1}}{\left( 1 - \frac{1}{\sqrt{\kappa}} \right) R_k^2 - \frac{2}{\alpha} \left( f(x_{k+1}^+) - f^* \right)}. \]

One way to ensure that $x_{k+1}$ satisfies the two key conditions, $f(x_{k+1}) \leq f(x_k^+)$ and inequality \eqref{eqn:separat}, is to simply let $x_{k+1}$ be the minimizer of $f$ along the line between $c_k$ and $x_k^+$.
Trivially this guarantees the inequality $f(x_{k+1}) \leq f(x_k^+)$, while the univariate optimality condition  $\nabla f(x_{k+1})\perp (c_k-x_{k+1})$ means the triangle with vertices $x_{k+1}$, $x_{k+1}^{++}$, and $c_k$ is a right triangle and inequality \eqref{eqn:separat} becomes ``the hypotenuse is longer than a leg.''
This is exactly the motivation for the line-search procedure in Algorithm~\ref{alg:optQuadAvg}. 
Repeating the process yields iterates $c_k$ that satisfy the optimal linear rate of convergence
\[ \tnorm{c_k - x^*}^2 \leq  \left( 1 - \frac{1}{\sqrt{\kappa}} \right)^k R_0^2. \] 
The precise method is described in Algorithm~\ref{alg:geoD}.

\begin{rem}
	{\rm
When applying an iterative method to compute $x_{k+1}=\ls{c_{k}}{x_{k}^+}$, one can use the following termination criterion. Check if $c_k$  satisfies $f(c_k)\leq f(x_k^+)$, then stop and set $x_{k+1}:=c_k$. Notice \eqref{eqn:separat} holds trivially with this choice of $x_{k+1}$. Else stop with a trial point $z$ on the line joining $c_{k}$ and $x_{k}^+$ satisfying $f(z)\leq f(x_k^+)$ and 
 \begin{equation*}
 \tnorm{z^{++} - c_k}^2 \geq \frac{\tnorm{\nabla f(z)}^2}{\alpha^2}.
 \end{equation*}
We claim that the line search will terminate in finite time, unless  $\ls{c_{k}}{x_{k}^+}$ is the true minimizer of $f$. Indeed, since $c_k\neq \ls{c_{k}}{x_{k}^+}$ (otherwise we would have terminated in the if clause), one can easily check that  $z=\ls{c_{k}}{x_{k}^+}$ satisfies the above inequality strictly.
}
\end{rem}

\begin{algorithm}
\caption{Geometric Descent Method [Bubeck, Lee, Singh]}  \label{alg:geoD}
\KwIn{Starting point $x_0$, strong convexity constant $\alpha>0$.}
\KwOut{ $x_K^+$ }
Set $c_0 = x_0^{++}$ and $R_0^2 = \frac{ \tnorm{\nabla f(x_0)}^2}{\alpha^2} - \frac{2}{\alpha} \left( f(x_0) - f(x_0^+) \right)$ \;
\For{k = 1, \ldots, K}{ 
Set $x_k = \ls{x_{k-1}^+}{c_{k-1}}$ \;
Set $x_A = x_k - \alpha^{-1} \nabla f(x_k)$ and $R_A^2 = \frac{ \tnorm{\nabla f(x_k)}^2}{\alpha^2} - \frac{2}{\alpha} \left( f(x_k) - f(x_k^+) \right)$ \;
Set $x_B = c_{k-1}$ and $R_B^2 = R_{k-1}^2 - \frac{2}{\alpha} \left( f(x_{k-1}^+) - f(x_k^+) \right)$ \;
Let $\B{c_k}{ R_k^2}$ be the 
smallest enclosing ball of $\B{x_A}{R_A^2} \cap \B{x_B}{ R_B^2}$ \;}
\end{algorithm}

The following theorem shows that Algorithm~\ref{alg:optQuadAvg} and Algorithm~\ref{alg:geoD} indeed produce the same iterate sequence.

\begin{thm}
Given the same initial point $x_0$, Algorithm~\ref{alg:optQuadAvg} and Algorithm~\ref{alg:geoD} produce the same iterates $x_k$ and $c_k$.
Moreover, we have $v_k = f(x_k^+) - \frac{\alpha}{2} R_k^2$, where $v_k$ is the minimum value of the quadratic $Q_k$ in Algorithm~\ref{alg:optQuadAvg} and $R_k$ is the radius of the ball in Algorithm~\ref{alg:geoD}.
\end{thm}

\begin{proof}
Let $x_k$ and $c_k$ denote the iterates in Algorithm~\ref{alg:optQuadAvg}, and let $\hat{x}_k$ and $\hat{c}_k$ be the iterates in Algorithm~\ref{alg:geoD}.
We proceed by induction on $k$.
It follows immediately from the definition of the algorithms that $x_0 = \hat{x}_0$, $c_0 = \hat{c}_0$, and $v_0 = f(x_0^+) - \frac{\alpha}{2} R_0^2$.
Now suppose, as an inductive assumption, $x_{k-1}=\hat{x}_{k-1}$, $c_{k-1}=\hat{c}_{k-1}$, and $v_{k-1}=f(x_{k-1}^+)-\frac{\alpha}{2}R_{k-1}^2$.
To see the equality $x_k = \hat{x}_k$, observe
\begin{align*}
x_k &= \ls{x_{k-1}^+}{c_{k-1}} = \ls{\hat{x}_{k-1}^+}{\hat{c}_{k-1}} = \hat{x}_k.
\end{align*}
Let $x_A = x_k^{++}$, $x_B = c_{k-1}$, $d = \norm{x_A - x_B}$, and define the quantities

\vspace{5pt}

\begin{tabular}{lll}
{ $\begin{aligned}
v_A &= f(x_k) - \frac{\norm{\nabla f(x_k)}^2}{2 \alpha}, \\
v_B &= v_{k-1}, \\
\end{aligned}$ } & &

{ $\begin{aligned}
R_A^2 &= \frac{\norm{\nabla f(x_k)}^2}{\alpha^2} - \frac{2}{\alpha} \left( f(x_k) - f(x_k^+) \right), \\
R_B^2 &= R_{k-1}^2 - \frac{2}{\alpha} \left( f(x_{k-1}^+) - f(x_k^+) \right). 
\end{aligned}$ }
\end{tabular}

\vspace{5pt}

Notice that $Q_k(x) = v_k + \frac{\alpha}{2} \norm{ x - c_k}^2$ is the optimal averaging of $Q_A(x) := v_A + \frac{\alpha}{2} \norm{x - x_A}^2$ and $Q_B(x) := v_B + \frac{\alpha}{2} \norm{x- x_B}^2$, and that $B(\hat{c}_k, R_k^2)$ is the minimum enclosing ball of the intersection of $B(x_A, R_A^2)$ and $B(x_B, R_B^2)$.
Simple algebra shows the relation 
\[ R_A^2 = \frac{2}{\alpha} \left( f(x_k^+) - v_A \right), \]
and from the inductive assumption $v_{k-1} = f(x_{k-1}^+) - \frac{\alpha}{2} R_{k-1}^2$, we also have 
\[ R_B^2 = \frac{2}{\alpha} \left( f(x_k^+) - v_B \right). \]
Thus, by Proposition~\ref{prop:rel_min_enc_opt_ave} and the discussion preceding it, we have $c_k = \hat{c}_k$ and $v_{k} = f(x_{k}^+) - \frac{\alpha}{2} R_{k}^2$. This completes the induction.
\end{proof}


As we saw in Section~\ref{sec:opt_quad_mem},  computing the optimal averaging of several quadratic functions is simple.
On the other hand, it is far from clear how to find the minimum radius ball that encloses the intersection of more than two balls.
Indeed, instead the authors of Algorithm~\ref{alg:geoD} in the follow-up work \cite{bubLeenew}  considered a ``relaxation'' that involves minimizing a self-concordant barrier for the intersection.
While revising the current manuscript, we became aware 
that Beck in \cite[Theorem 3.2]{beck_smallball} proved that the minimum enclosing ball of the intersection of finitely many balls can be computed by solving a convex quadratic program (QP). Namely, Beck showed that the squared radius of the minimal ball enclosing the intersection $\bigcap^t_{i=1} B(c_i,r_i^2)$ is exactly equal to 
$$\min_{\lambda\in \Delta_t}~ \left\|\sum_{i=1}^t \lambda_i c_i\right\|^2-\sum_{i=1}^t \lambda_{i}(\left\|a_i\right\|^2-r_i^2),$$
provided $t\leq n-1$ and the intersection of the balls has nonempty interior. This QP is exactly the one we derived in Section~\ref{sec:opt_quad_mem} for the optimal quadratic averaging method with memory. Note that our derivation  of the QP in  Section~\ref{sec:opt_quad_mem} was completely elementary; the proof of \cite[Theorem 3.2]{beck_smallball}, on the other hand, is much more sophisticated relying on an S-lemma-type result.


\begin{prop}[Optimal quadratic averaging \& minimal enclosing ball] {\hfill \\ }
Let $Q(x) = v + \frac{\alpha}{2} \norm{x - c}^2$ be the optimal averaging of quadratics $Q_i(x) = v_i + \frac{\alpha}{2} \norm{x - c_i}^2$ for $i = 1, \ldots, t$ with $t<n$. 
Fix a real number $s\geq v_i$ for all $i=1\ldots,t$ and define the balls $B_i := \{ Q_i \leq s \}$. Then provided that the intersection $\bigcap_{i=1}^t B_i$ has a nonempty interior, the ball $B := \{ Q \leq s \}$ is the minimal enclosing ball of the intersection $\bigcap_{i=1}^t B_i$.
\end{prop}
\begin{proof}
Let $R^2$ be the square radius of $B$ and let $R_i^2$ be the square radius of $B_i$, for $i = 1, \ldots, t$.
Using Proposition~\ref{prop:optavg}, we deduce
\begin{align*}
R^2 = \frac{2}{\alpha} (s-v)
&= \frac{2}{\alpha}\left(s-\max_{\lambda\in\Delta_t}~\left\{ \frac{\alpha}{2}\sum_{i=1}^t\lambda_i\left(\frac{\alpha}{2}\|c_i\|^2+v_i\right) -\frac{\alpha}{2}\left\|\sum_{i=1}^t \lambda_ic_i\right\|^2\right\}\right)\\
&=\min_{\lambda\in\Delta_t} ~\norm{\sum_{i=1}^t \lambda_i c_i}^2 - \sum_{i=1}^t \lambda_i \left( \norm{c_i}^2 + \frac{2}{\alpha} (v_i - s) \right) \\ 
&= \min_{\lambda\in\Delta_t}~\norm{\sum_{i=1}^t \lambda_i c_i}^2 - \sum_{i=1}^t \lambda_i \left( \norm{c_i}^2 - R_i^2 \right).
\end{align*}
The center of $B$ is $c =  \sum_{i=1}^t \lambda_i c_i$ where $\lambda$ is the minimizer of the expression above.
Comparing with \cite[Theorem 3.2]{beck_smallball}, we see that $B$ is exactly the minimum radius ball enclosing the intersection $\bigcap_{i=1}^t B_i$.
\end{proof}




\section{Numerical examples} \label{sec:numerical_examples}
In this section, we numerically illustrate optimality gap convergence in Algorithm~\ref{alg:optQuadAvg}, and explore how Algorithm~\ref{alg:optAvgWithMem}, the variant of Algorithm~\ref{alg:optQuadAvg} with memory, aids performance.
To this end, we focus on minimizing two functions: the regularized logistic loss function
\[ L(w) := \frac{1}{N} \sum_{i=1}^N \log \left( 1 + e^{-y_i w^T x_i} \right) + \frac{\alpha}{2} \tnorm{w}^2, \]
where $x_i \in \R^n$ and $y_i \in \{ \pm 1\}$ are labeled training data, and the ``world's worst'' function for first-order methods:
\[ f(x) = \frac{B}{2} \left( (1-x_1)^2 + \sum_{i=1}^{n-1} (x_i - x_{i+1})^2 + x_n^2 \right) + \frac{1}{2} \sum_{i=1}^n x_i^2 \]
(see \cite[Section 2.1.2 and Section 2.1.4]{nestBook}).
For the logistic regression examples, we use the LIBSVM \cite{libsvm}  data sets a1a ($N = 1605$, $n = 123$) and colon-cancer ($N = 62$, $n = 2000$).

\subsection{Optimality gap convergence}
From inequality \eqref{eqn:lowerBound}, we get the well-known optimality gap estimate for strongly convex functions
\begin{equation} \label{eqn:optGapEst}
f(x) - f^*\leq \frac{\tnorm{\nabla f(x)}^2 }{2 \alpha}.
\end{equation}
How does this estimate compare with the gaps $g_k := f(x_k^+) - v_k$ generated by Algorithm~\ref{alg:optQuadAvg}?
Obviously the answer depends on the point where we evaluate the gap estimate in~\eqref{eqn:optGapEst}.
Nonetheless, we can say that the gaps $g_k$ are tighter than the gaps $G_k := \frac{\tnorm{\nabla f(x_k) }^2 }{2 \alpha}$.
Indeed, by the definition of $v_k$, we trivially have $v_k \geq f(x_k) - G_k$ and thus
\[ g_k = f(x_k^+) - v_k \leq f(x_k) - v_k \leq G_k. \] 
On a relative scale, the difference between $g_k$ and $G_k$ is striking; see Figure~\ref{fig:relGap}.
Notice that $G_k$ is an optimality gap estimate before averaging, and $g_k$ is an optimality gap estimate after averaging; the plots in Figure~\ref{fig:relGap} show that optimal quadratic averaging makes great relative progress per iteration.

\begin{figure}[h!]
\begin{center}
\includegraphics[width = 3in]{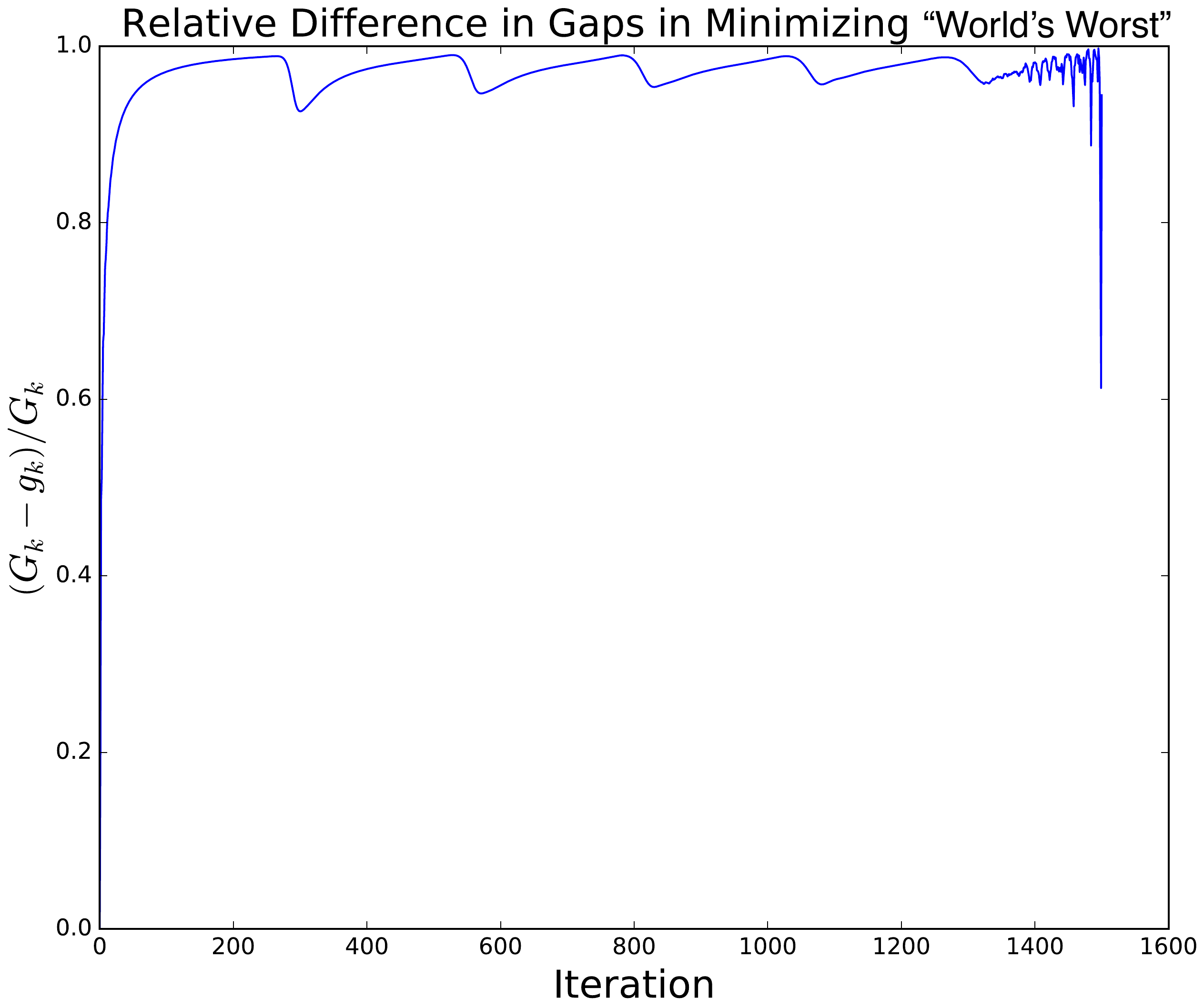} \includegraphics[width = 3in]{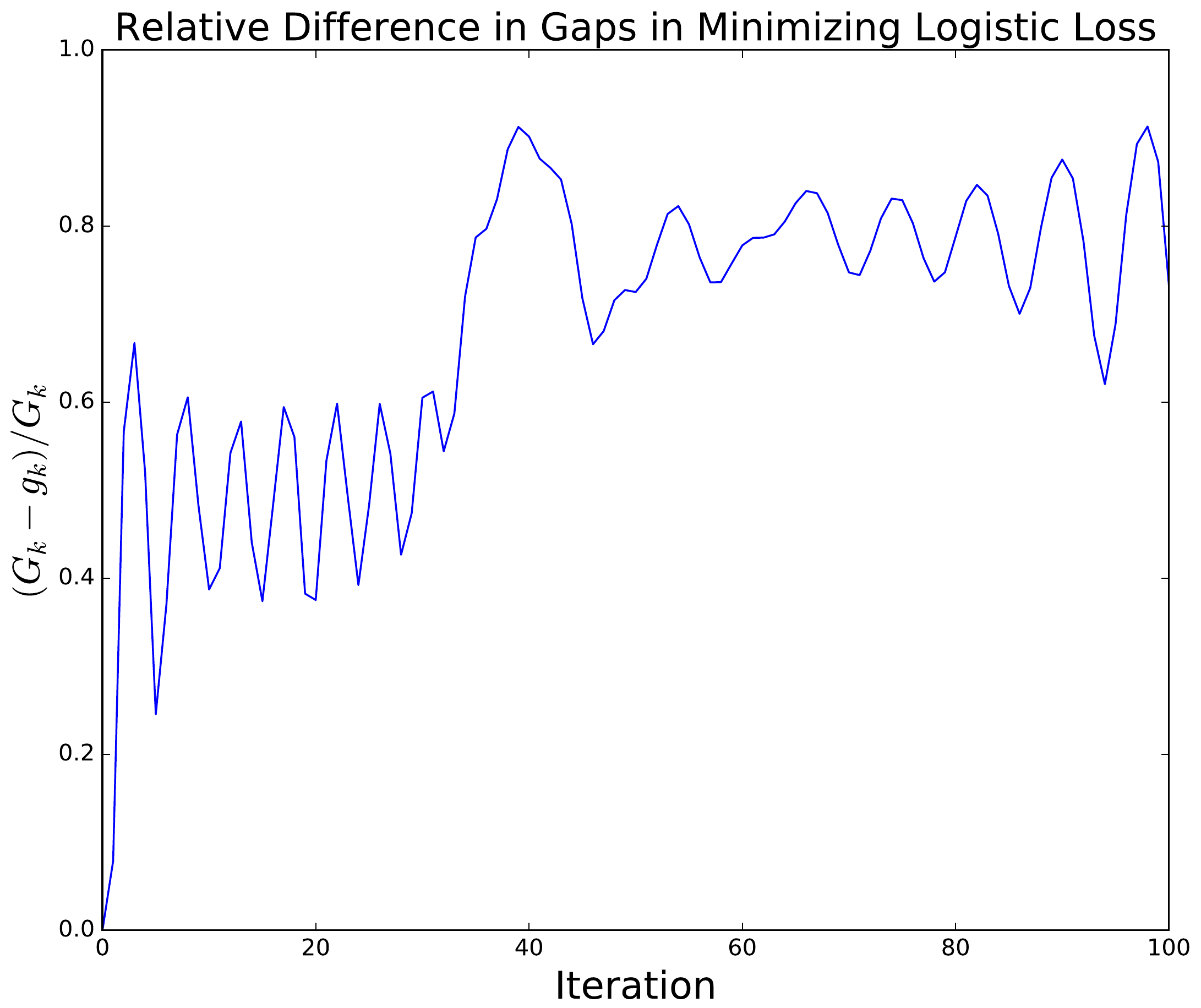}
\end{center}
\caption{Relative differences in gaps $\frac{G_k - g_k}{G_k}$ on the ``world's worst'' function ($B = 10^6$, $n = 200$), and on the logistic loss on the colon-cancer data set with regularization $\alpha = 0.0001$.} \label{fig:relGap}
\end{figure}

In Figure~\ref{fig:absGaps}, we plot $g_k$, the true gaps $f(x_k^+) - f^*$, and the gap estimate in~\eqref{eqn:optGapEst} at $x_k$, $x_k^+$, and $c_k$ for the ``world's worst'' function and the logistic loss function.
The true gaps are the tightest, albeit unknown at runtime.
Surprisingly, the gaps $\frac{\tnorm{ \nabla f(c_k)}^2}{2 \alpha}$ are quite bad: several orders of magnitude larger than $g_k$.
So even though the centers $c_k$ may appear to be the focal points of the algorithm, the points $x_k^+$ are the ones to monitor in practice.
Finally we note that the gaps $g_k$ and $\frac{\tnorm{ \nabla f(x_k^+) }^2}{2 \alpha}$ are comparable, even though $g_k$ does not rely on gradient information at $x_k^+$.

\begin{figure}[h!]
\begin{center}
\includegraphics[width = 3in]{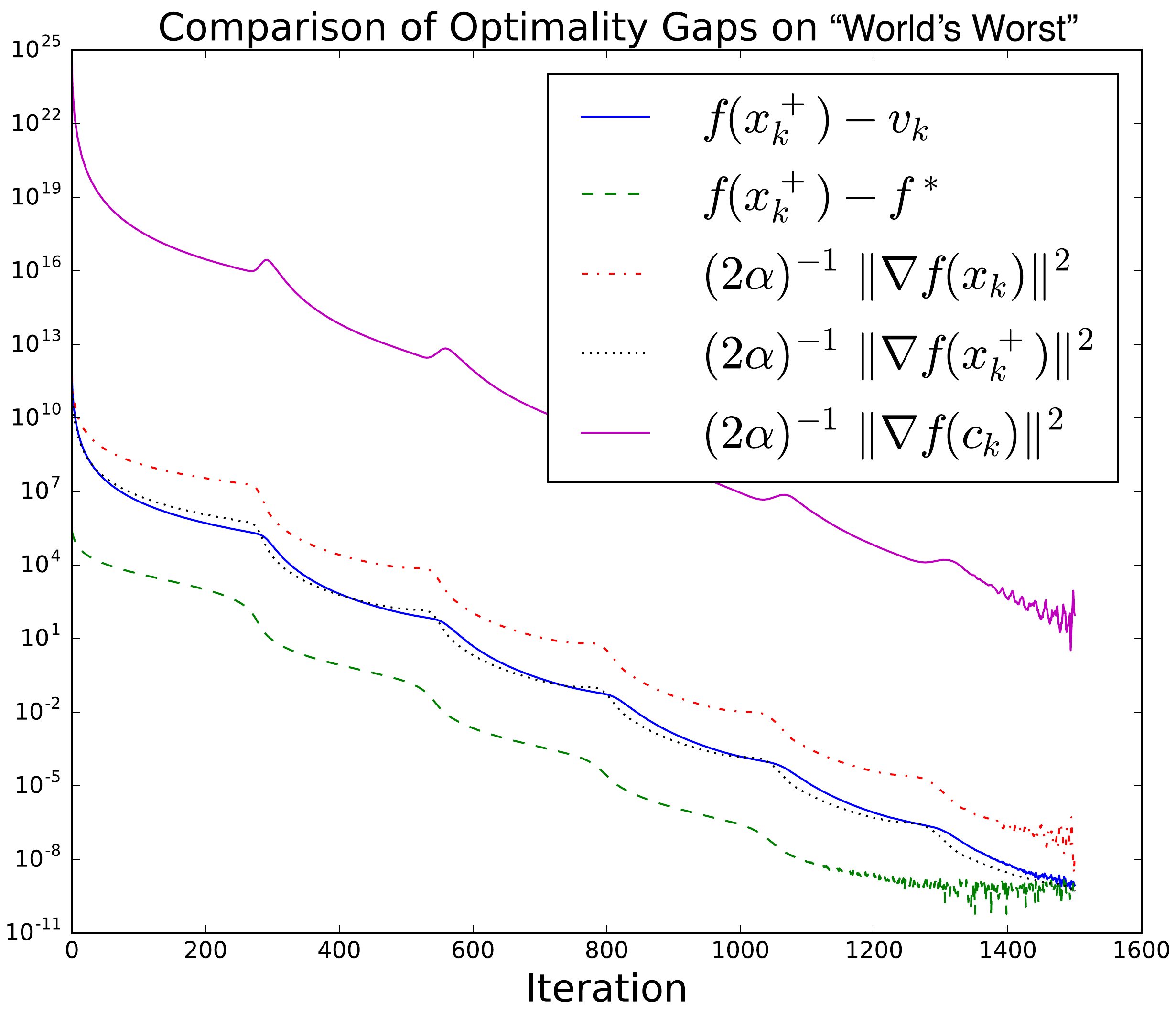} \includegraphics[width = 3in]{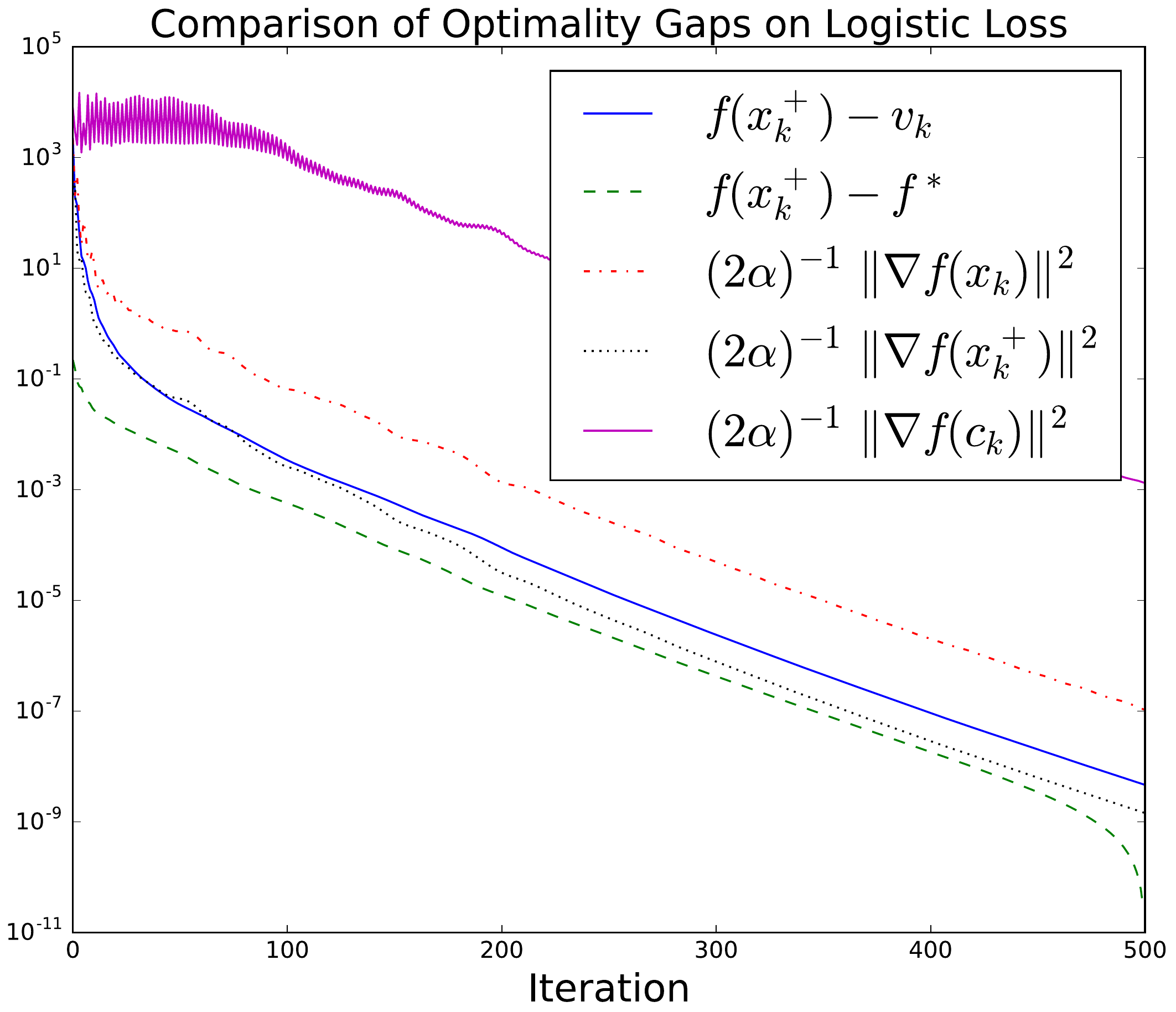}
\end{center}
\caption{Comparison of various optimality gaps on the ``world's worst'' function ($B = 10^6$, $n = 200$), and on the logistic loss on the a1a data set with regularization $\alpha = 0.0001$.} \label{fig:absGaps}
\end{figure}

\subsection{Optimal quadratic averaging with memory}
To demonstrate the effectiveness of optimal quadratic averaging with memory, we use it to minimize the logistic loss (see Figure~\ref{fig:logistic}).
The speedup over the memoryless method is significant, even when taking into account the extra work per iteration needed to solve the small dimensional quadratic subproblems.
In Figure~\ref{fig:lbfgs}, we compare Algorithm~\ref{alg:optAvgWithMem} with L-BFGS.
The two schemes are on par with each other, and neither is better than the other in all cases.

\begin{figure}[h!]
\begin{center}
\includegraphics[width=3in]{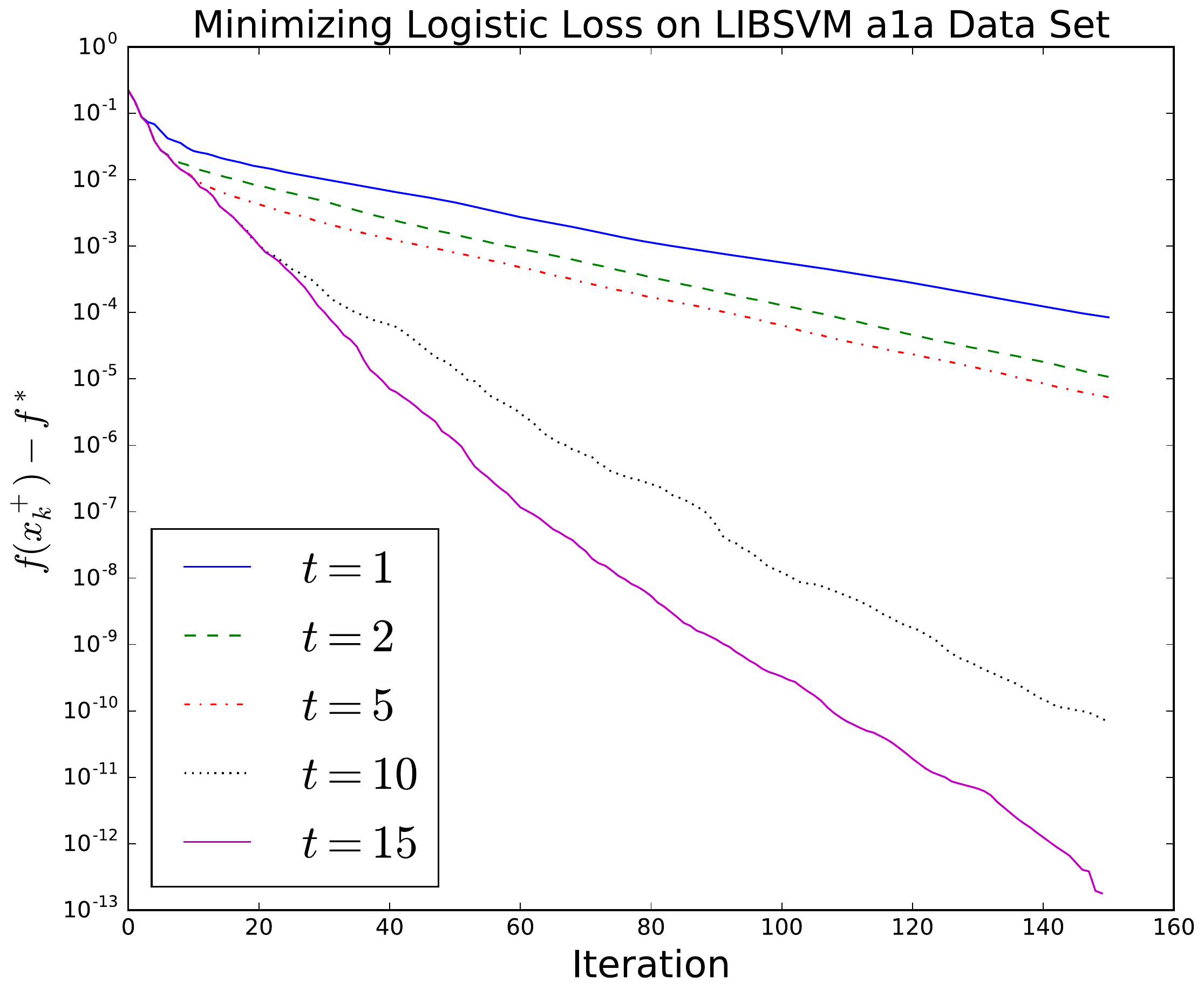} \includegraphics[width = 3in]{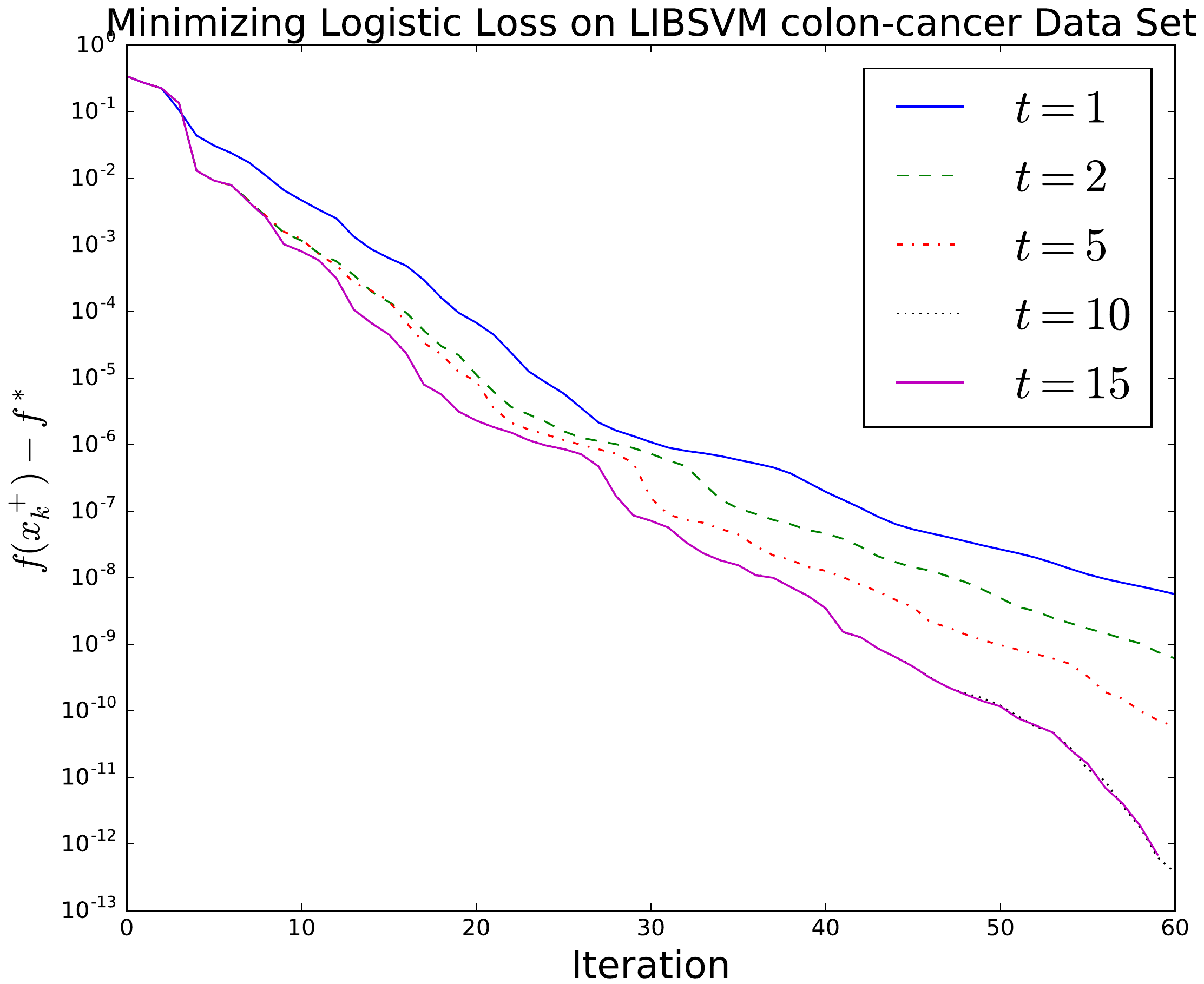}
\end{center}
\caption{Algorithm~\ref{alg:optAvgWithMem} with various memory sizes $t$.  The case $t = 1$ corresponds to the memoryless optimal averaging method in Algorithm~\ref{alg:optQuadAvg}.  The task is logistic regression, with regularization $\alpha = 0.0001$, on data sets a1a and colon-cancer.} \label{fig:logistic}
\end{figure}

\begin{figure}[h!]
\begin{center}
\includegraphics[width=3in]{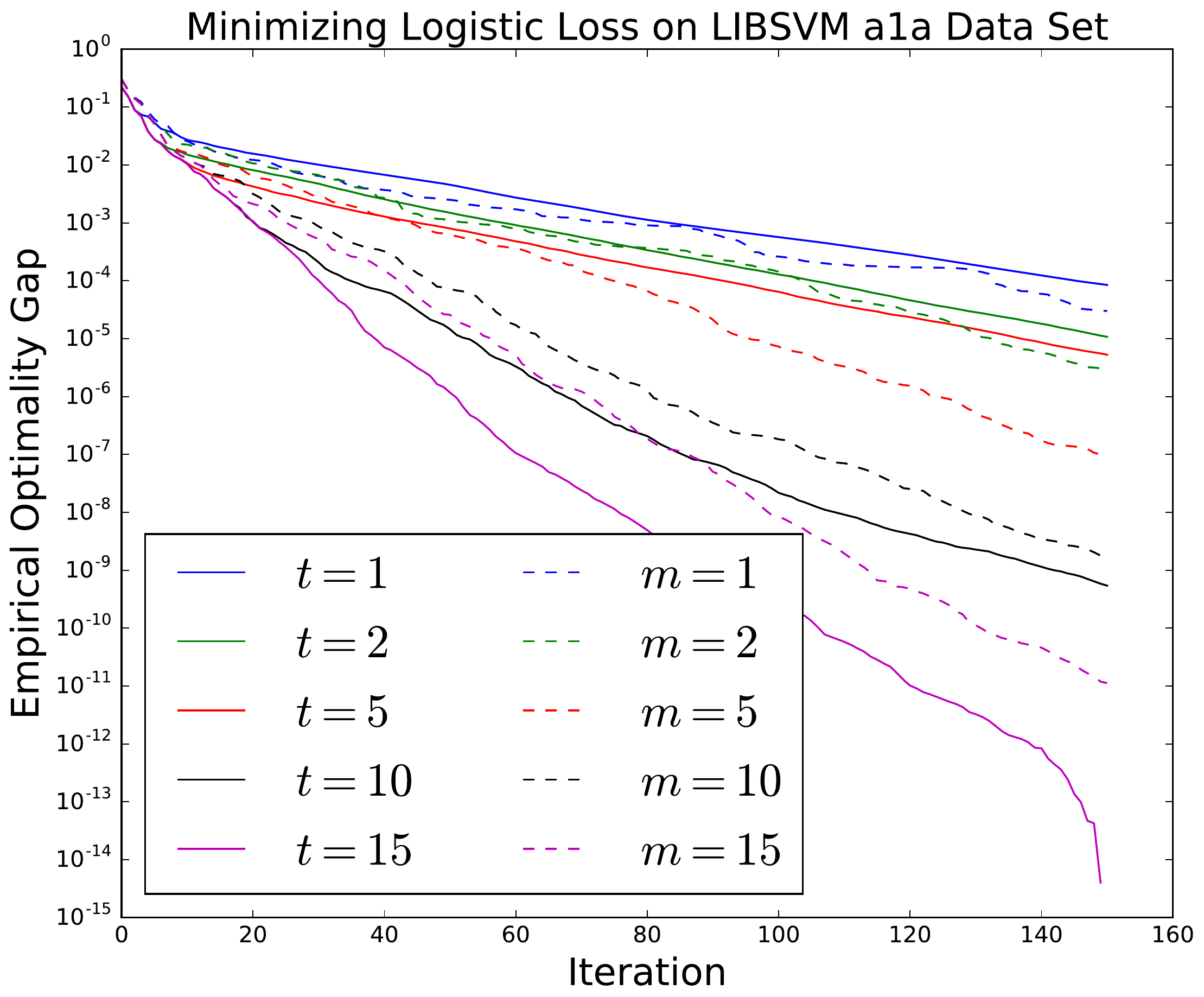} \includegraphics[width = 3in]{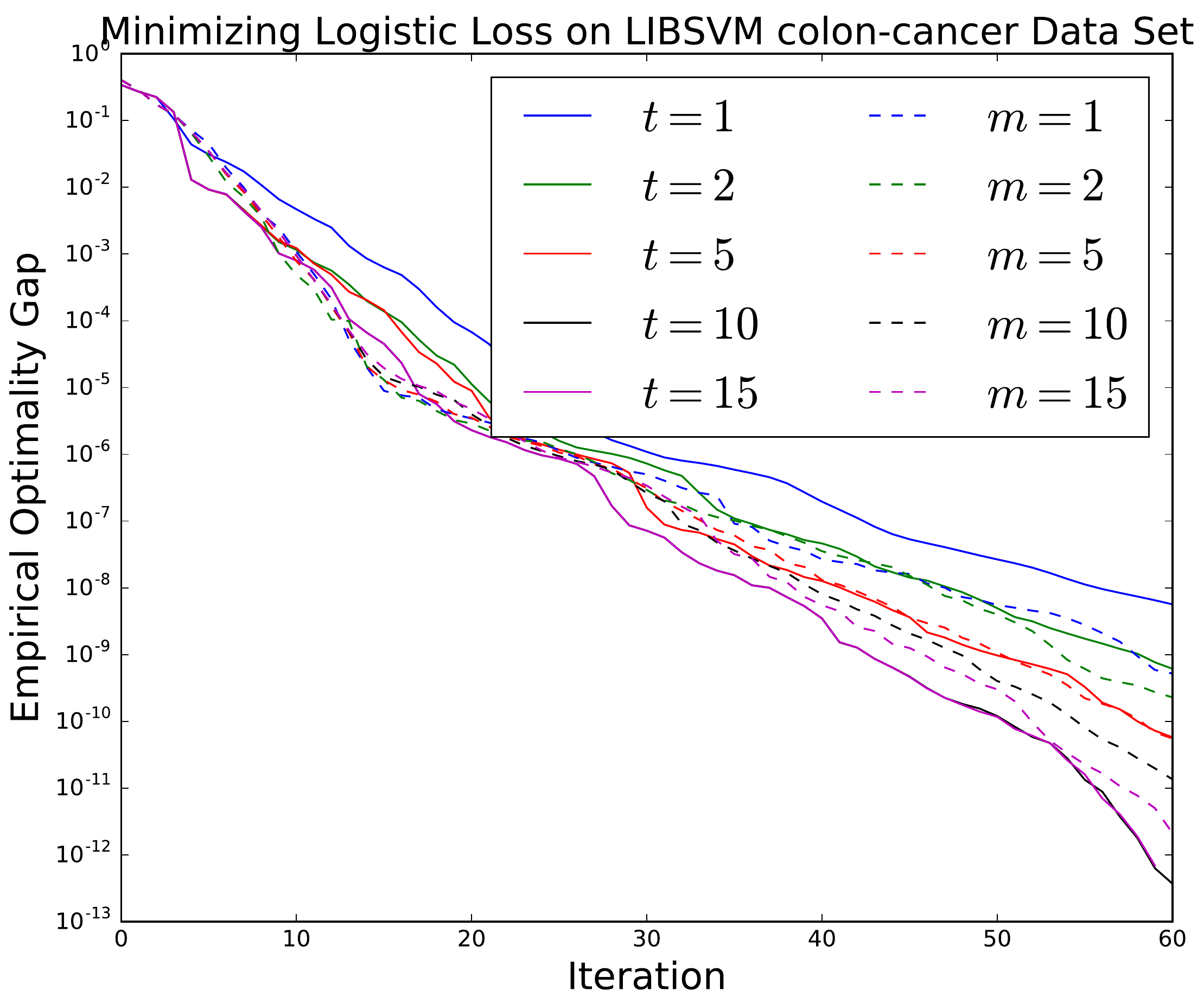}
\end{center}
\caption{Algorithm~\ref{alg:optAvgWithMem} with memory size $t$ versus L-BFGS with memory size $m$.  The task is logistic regression, with regularization $\alpha = 0.0001$, on data sets a1a and colon-cancer.} \label{fig:lbfgs}
\end{figure}

It is perhaps fairer to compare L-BFGS with memory size $m$ to Algorithm~\ref{alg:optAvgWithMem} with memory size $t = 2m$ (see Figure~\ref{fig:logistic_same_mem}).
Indeed, L-BFGS with memory size $m$ actually stores $m$ \emph{pairs} of vectors, whereas Algorithm~\ref{alg:optAvgWithMem} with memory size $t$ only stores $t$ vectors.
Moreover, the most expensive operation per iteration in L-BFGS requires $4mn$ multiplications (see \cite[Algorithm 7.4]{NW}); in contrast, computing a new center in Algorithm~\ref{alg:optAvgWithMem} requires $2 n (t + 1)$ multiplications plus the cost of solving a small quadratic program.
(Updating the matrix $C^T C$ takes $t+1$ inner products in $\R^n$, finding $\lambda$ amounts to solving a small quadratic program, and computing $C \lambda$ takes $n$ inner products in $\R^{t+1}$.)
In Figure~\ref{fig:logistic_smaller_alpha}, we again compare L-BFGS and Algorithm~\ref{alg:optAvgWithMem} on logisitic regression, but with less regularization.

\begin{figure}[h!]
\begin{center}
\includegraphics[width=3in]{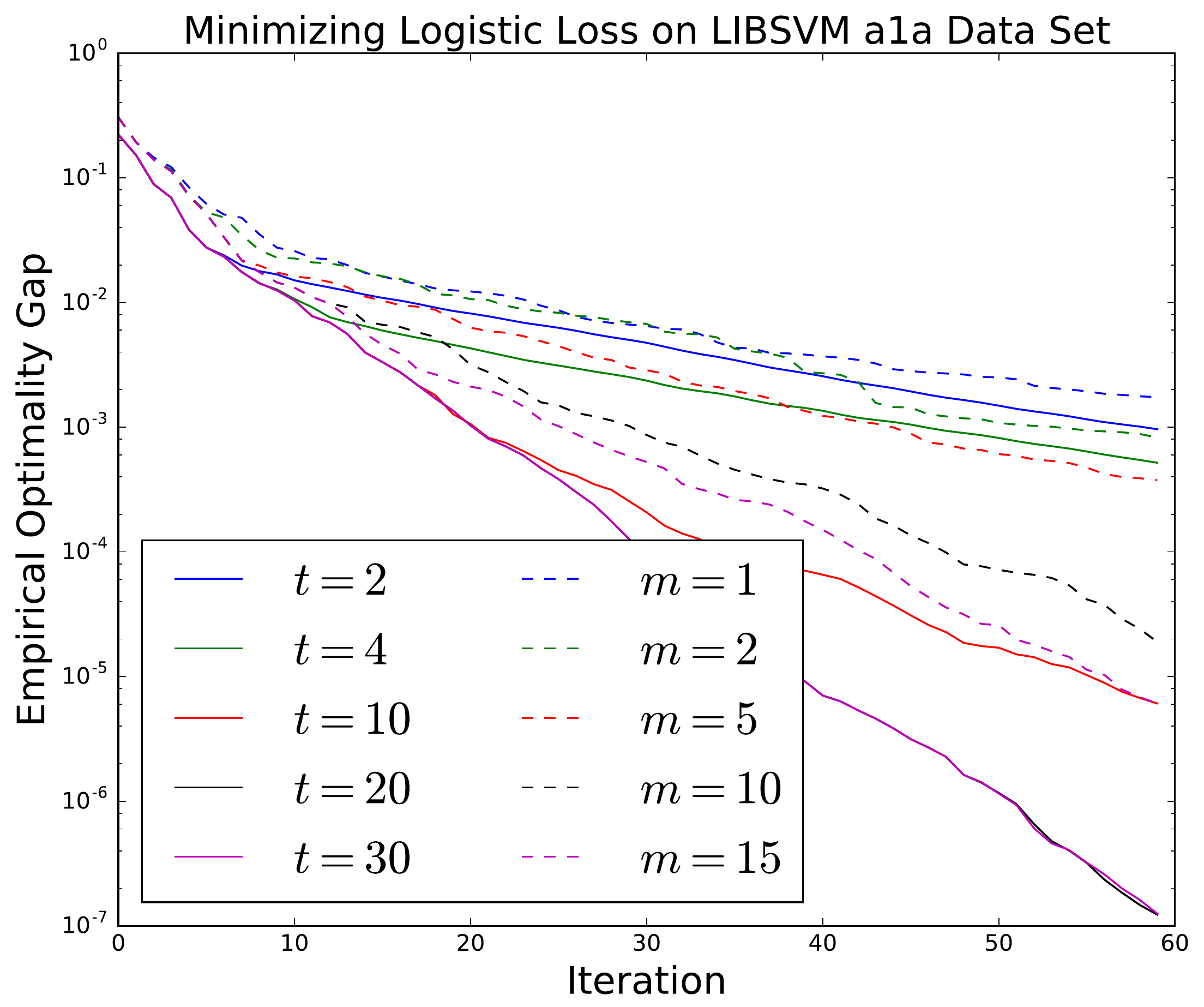} \includegraphics[width = 3in]{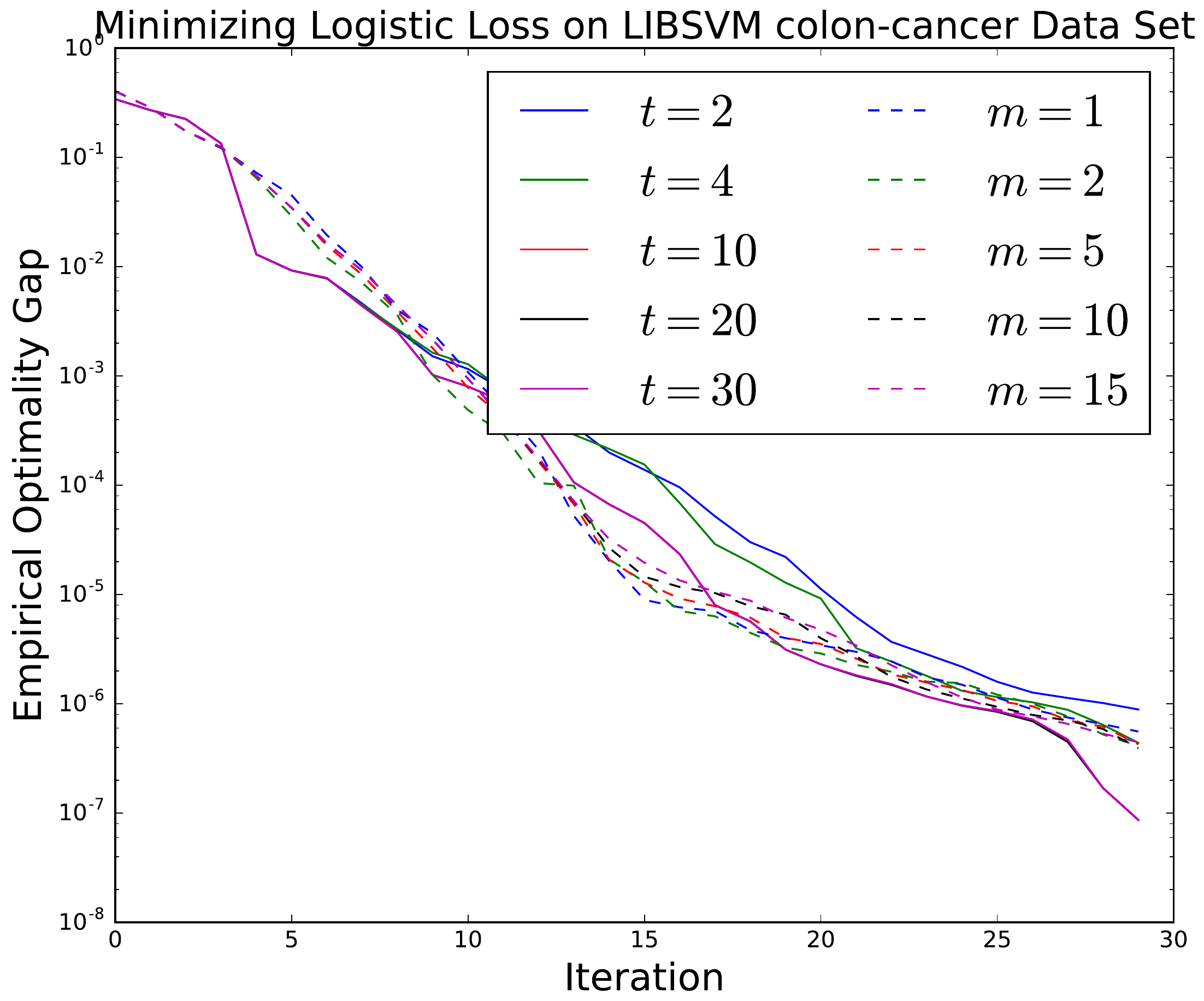}
\end{center}
\caption{A fairer (equal memory) comparison of Algorithm~\ref{alg:optAvgWithMem} and L-BFGS. The task is still logistic regression, with regularization $\alpha = 0.0001$, on data sets a1a and colon-cancer. We focus on lower accuracy than we did in Figure~\ref{fig:lbfgs}.  } \label{fig:logistic_same_mem}
\end{figure}

\begin{figure}[h!]
\begin{center}
\includegraphics[width=3in]{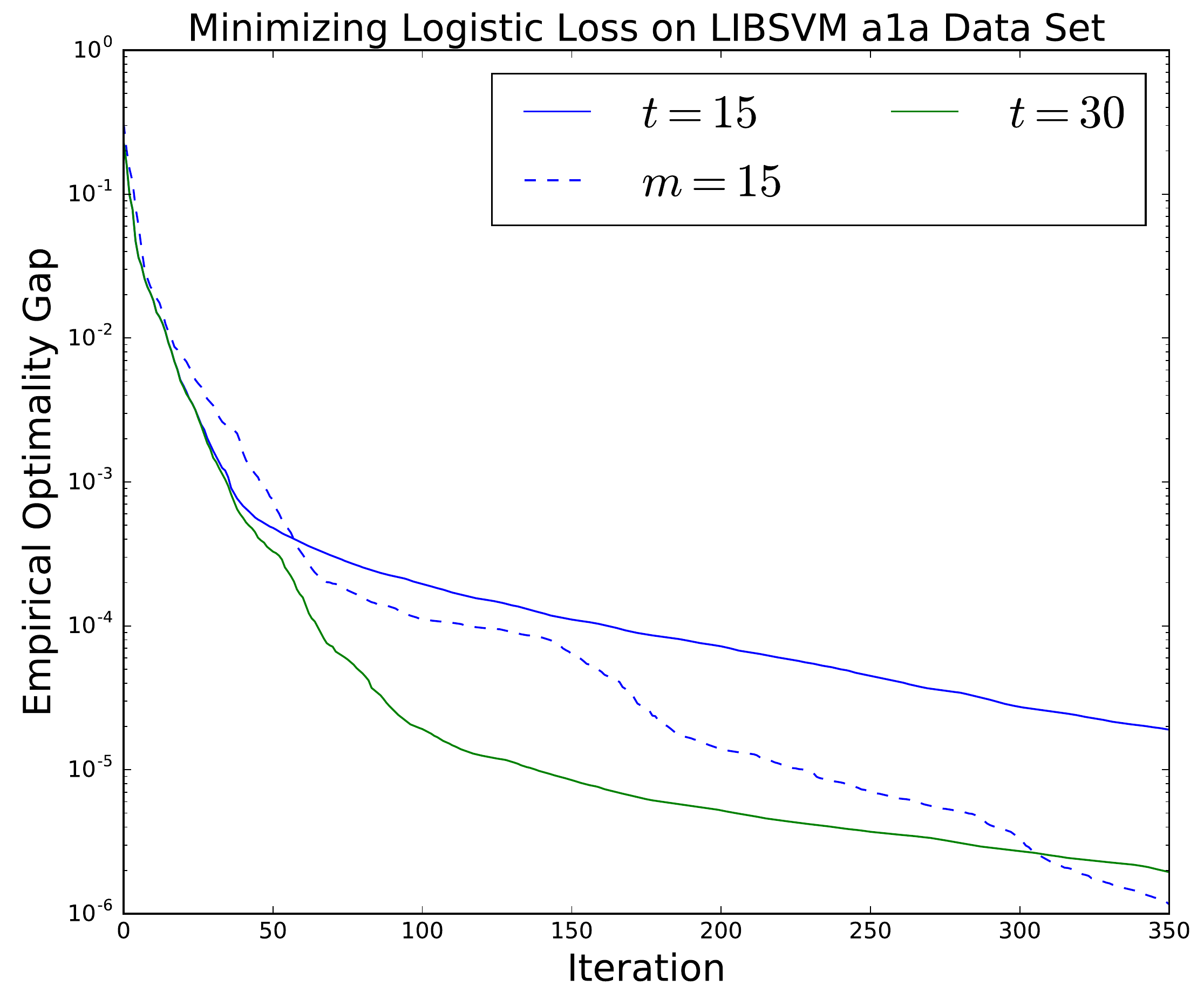} \includegraphics[width = 3in]{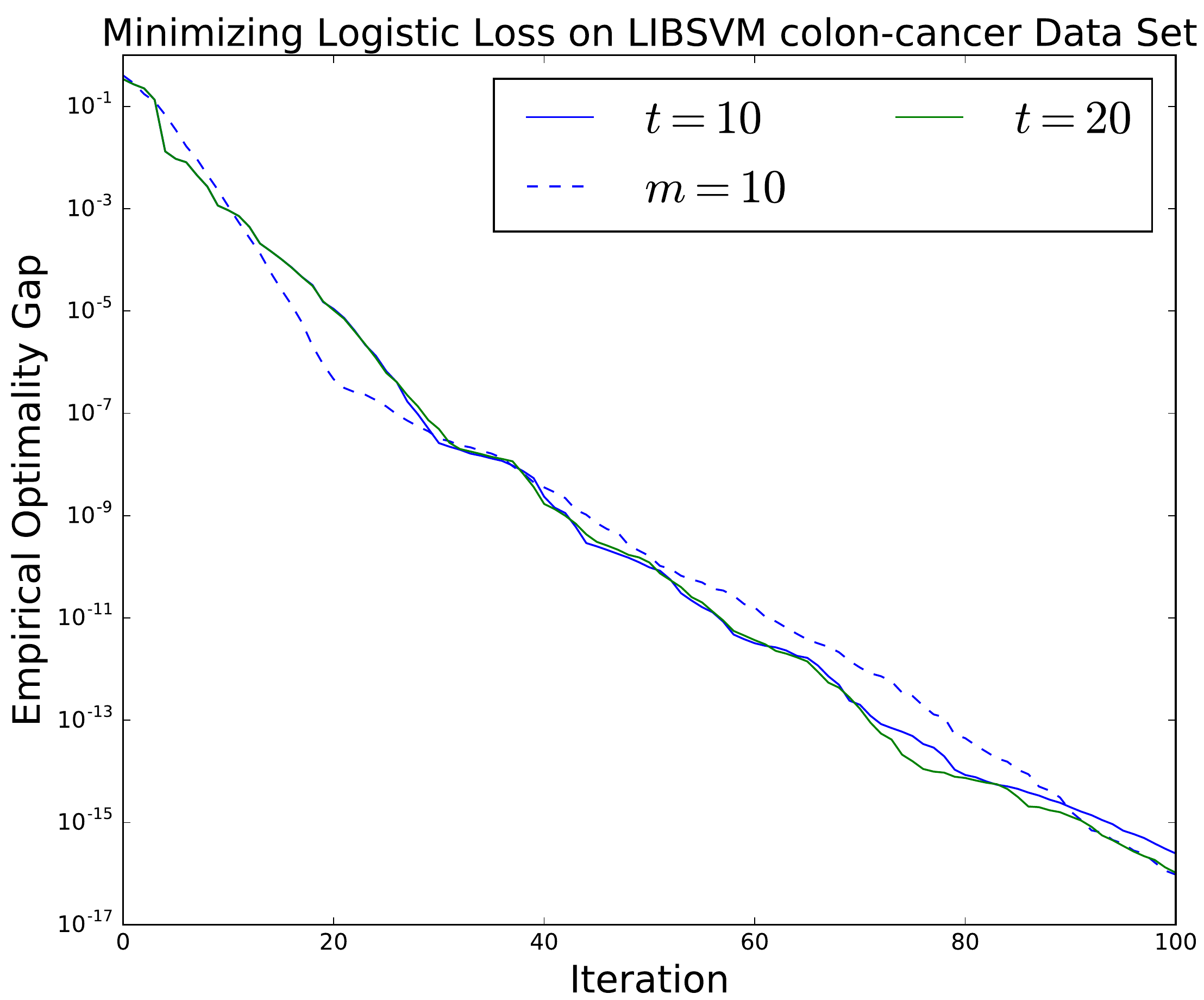}
\includegraphics[width=3in]{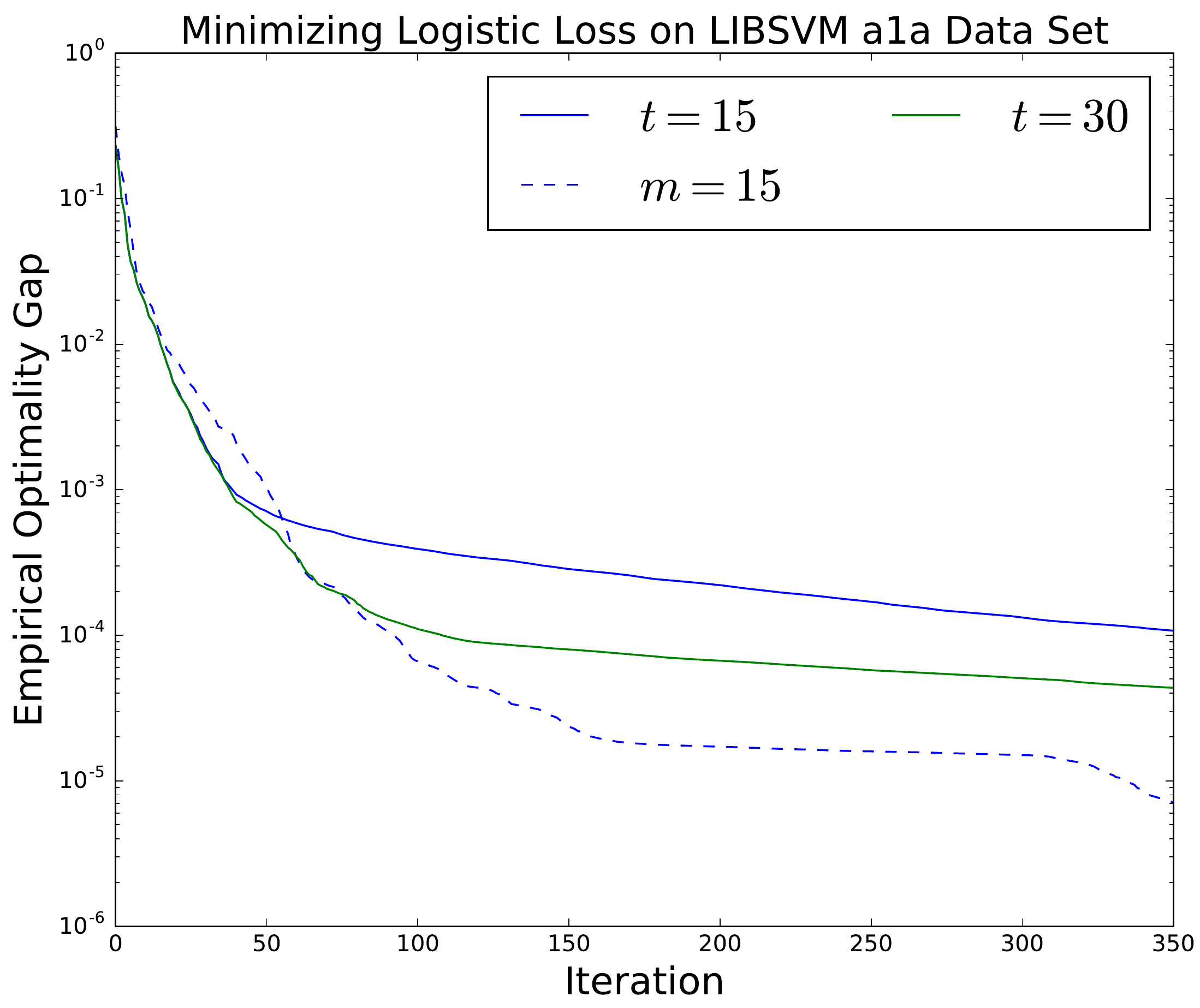} \includegraphics[width = 3in]{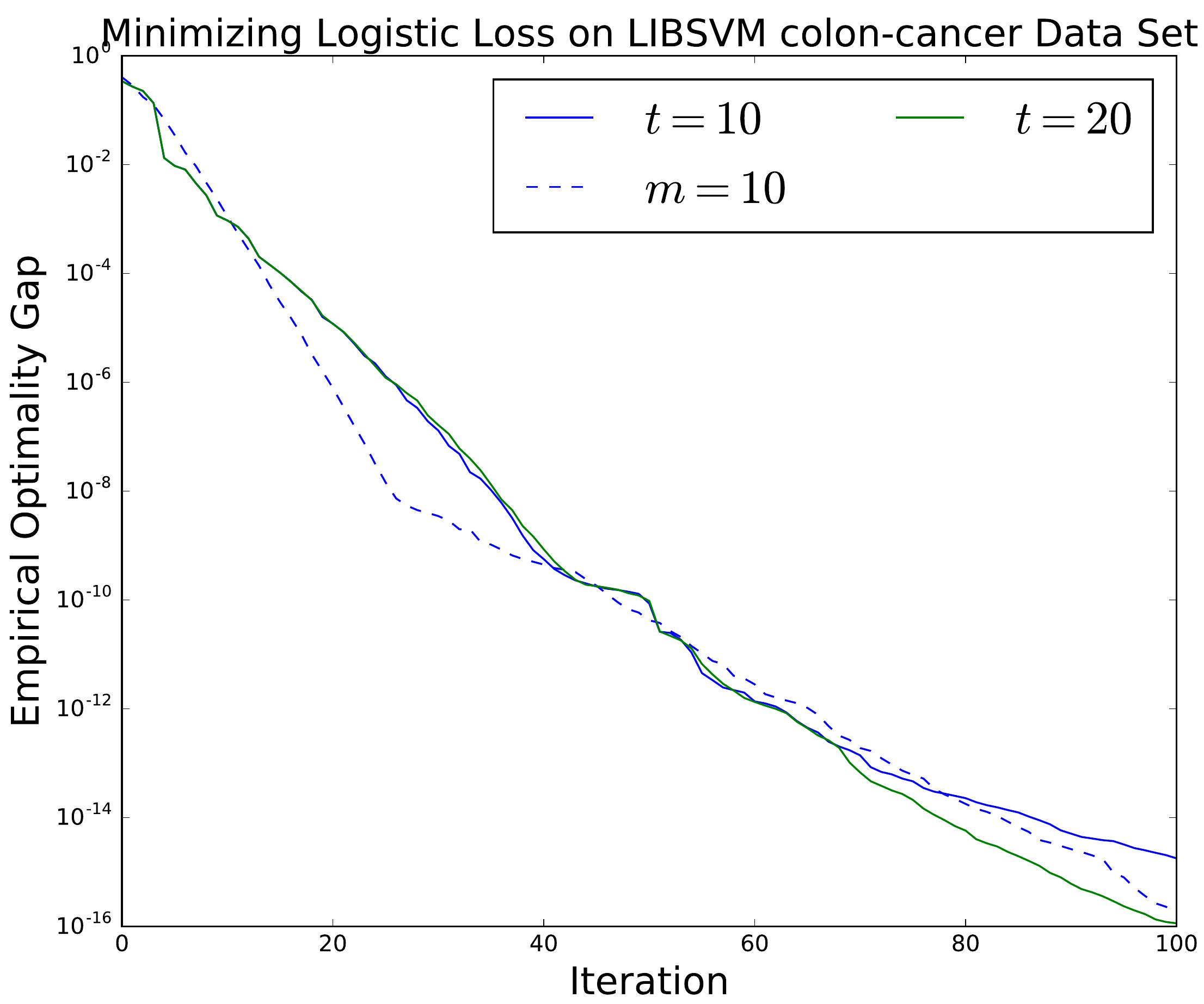}
\end{center}
\caption{Algorithm~\ref{alg:optAvgWithMem} with memory size $t$ versus L-BFGS with memory size $m$.  The task is logistic regression on data sets a1a and colon-cancer, with $\alpha = 10^{-6}$ (top row) and $\alpha = 10^{-8}$ (bottom row).} \label{fig:logistic_smaller_alpha}
\end{figure}

We noticed that the small dimensional quadratic program in Algorithm~\ref{alg:optAvgWithMem} must be solved to high accuracy, especially on poorly conditioned problems; an active-set method works well.
Accuracy in the line search is less important.
Minimizing the one-dimensional function $r \mapsto f(x + r d)$, with $\norm{d} = 1$, to within $10^{-4}$ accuracy in $r$ works well in general.
In Figure~\ref{fig:ls_tol}, we show how line search accuracy affects Algorithm~\ref{alg:optQuadAvg}.

\begin{figure}[h!]
\begin{center}
\includegraphics[width=3in]{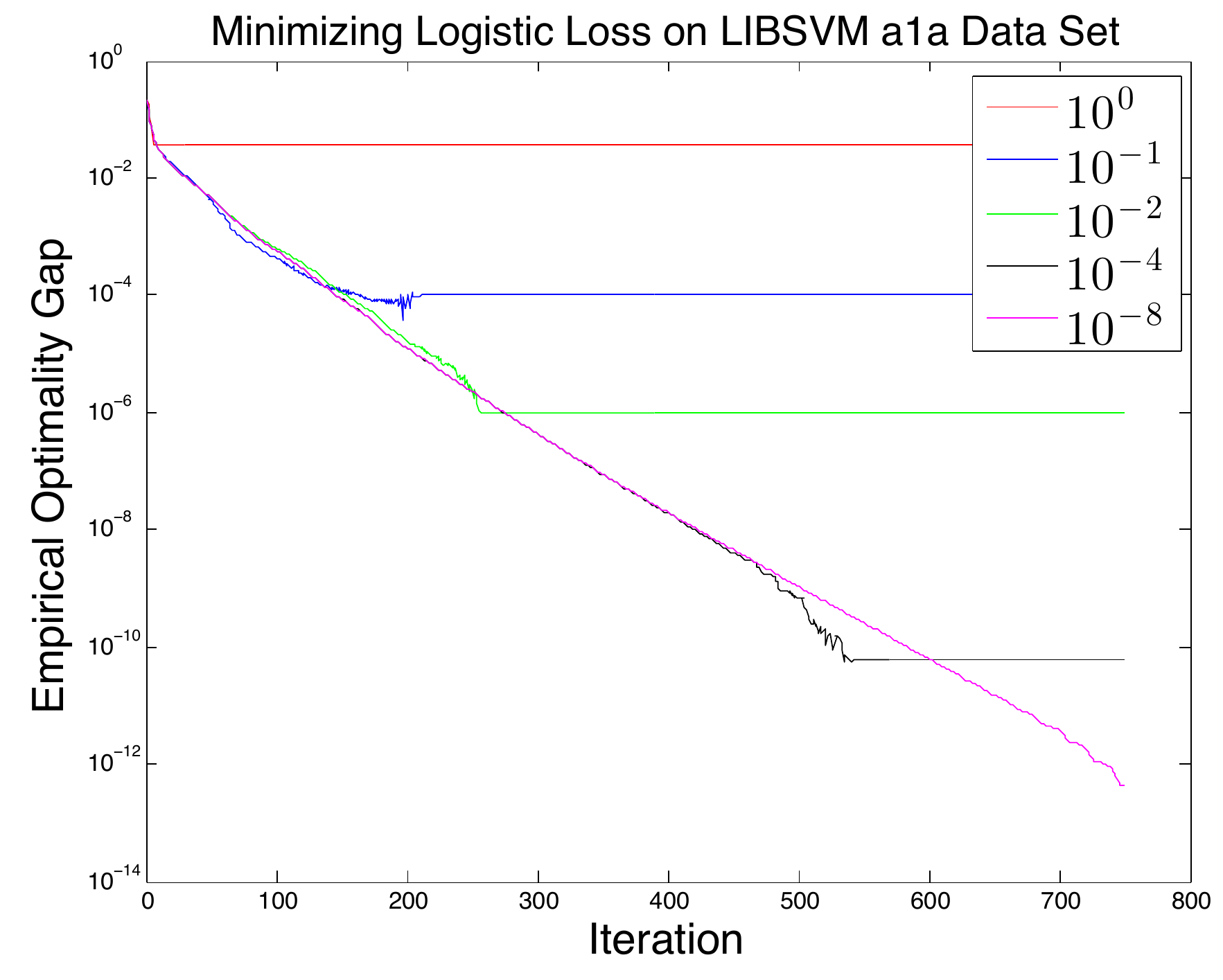} \includegraphics[width = 3in]{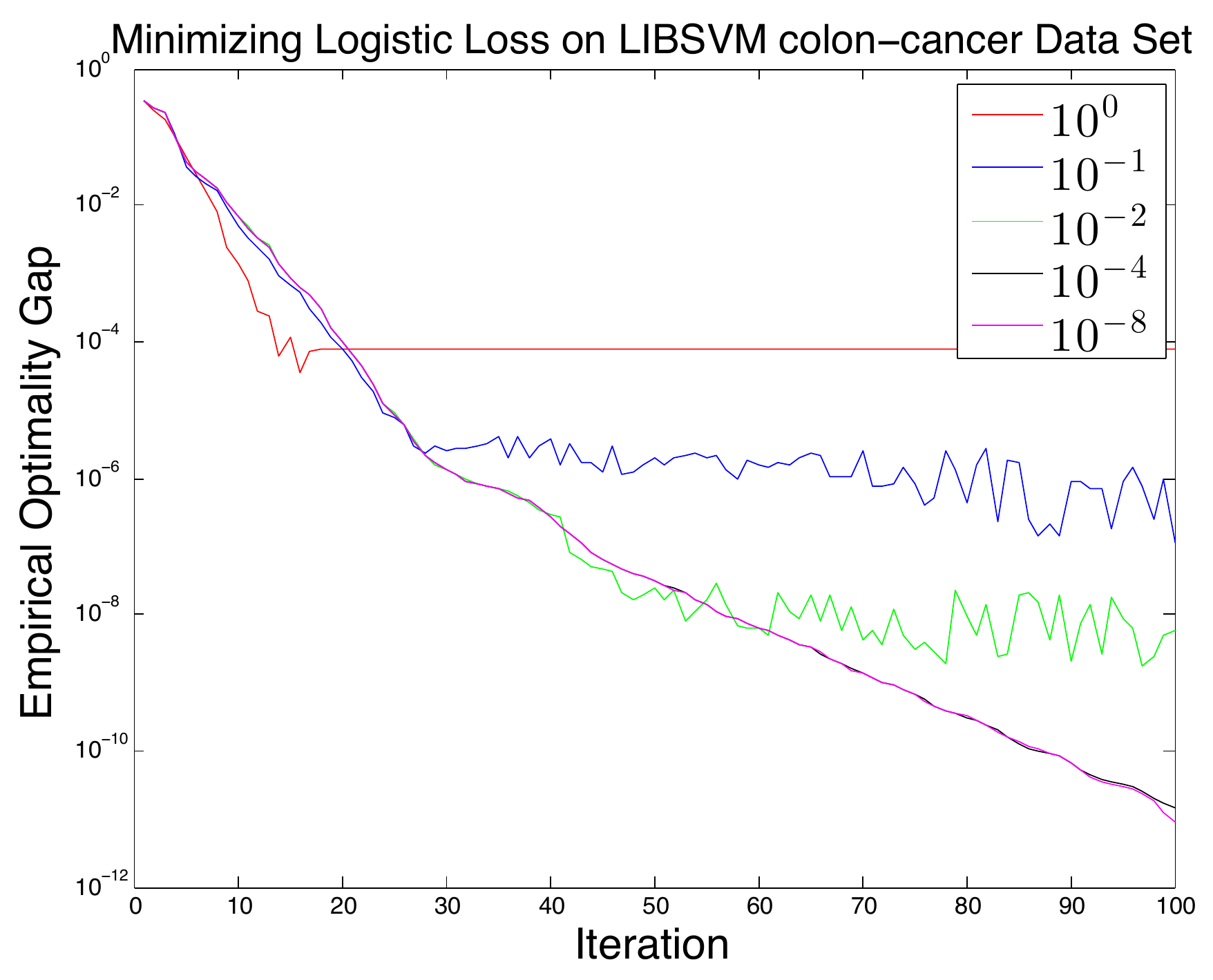}
\includegraphics[width=3in]{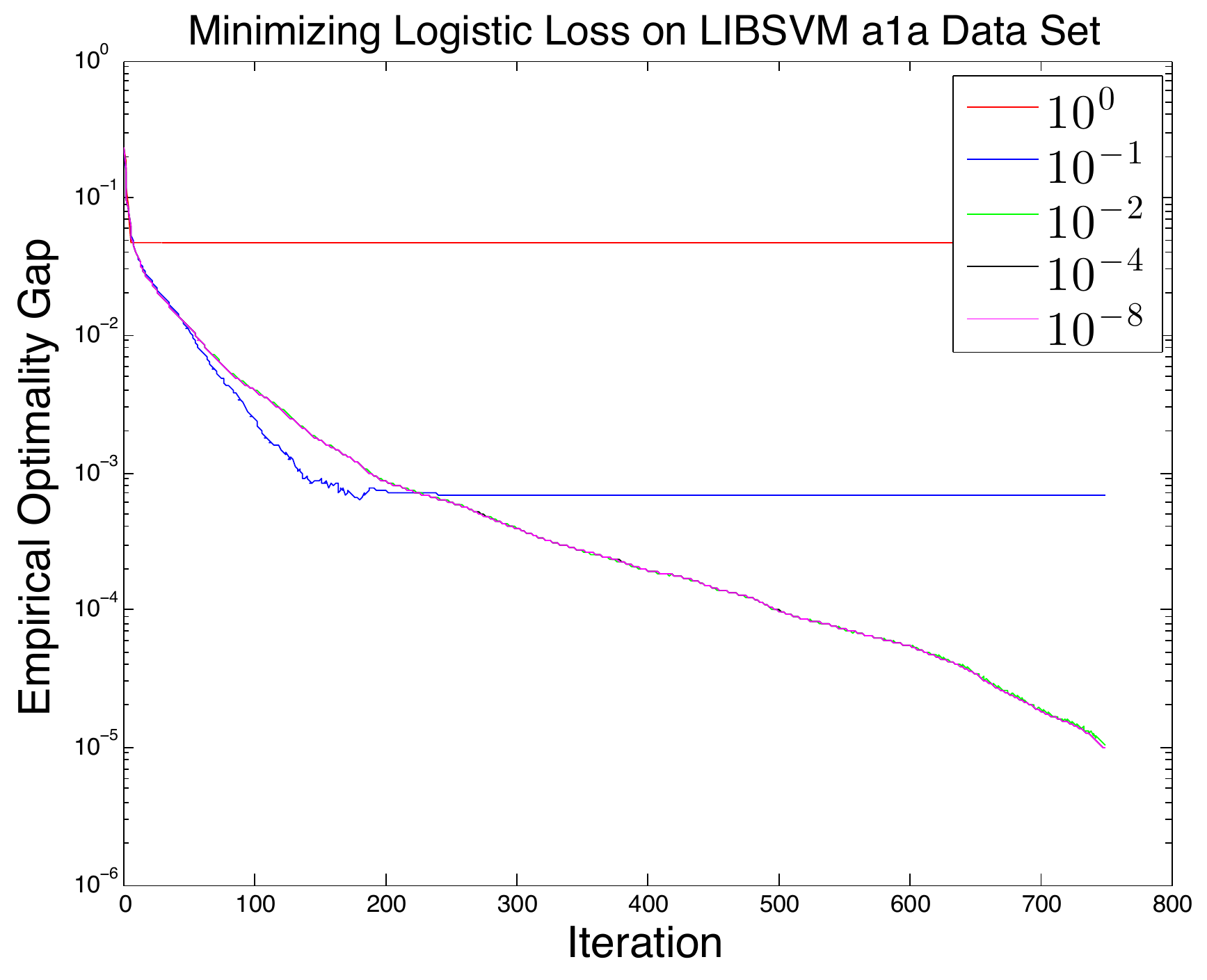} \includegraphics[width = 3in]{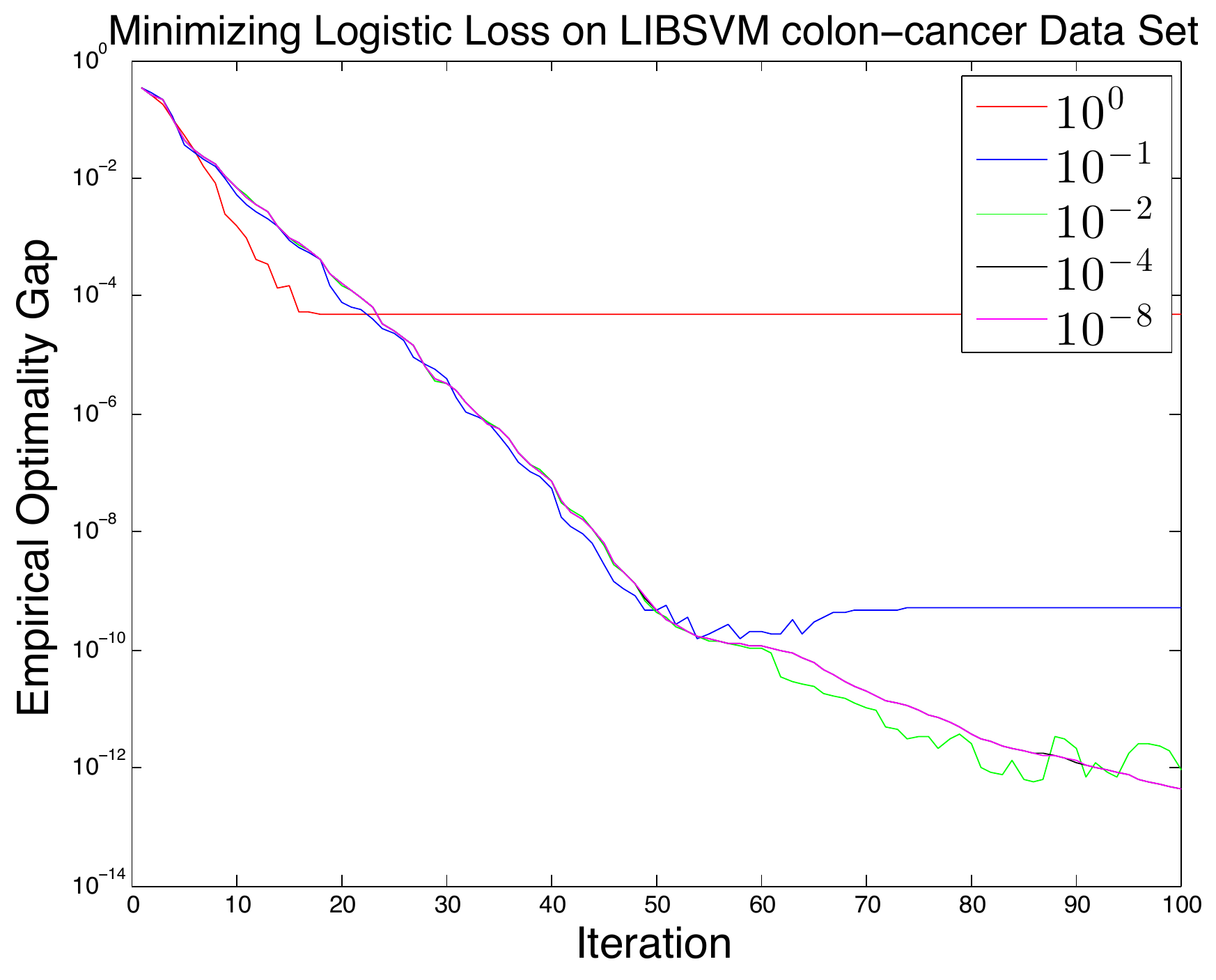}
\end{center}
\caption{A comparison of how the line search tolerance in Algorithm~\ref{alg:optQuadAvg} affects convergence.  In the top row, we do the comparison with logistic regression on the a1a and colon-cancer data sets with regularization $\alpha = 10^{-4}$.  In the bottom row, we use regularization $10^{-8}$.} \label{fig:ls_tol}
\end{figure}

	\newpage
\section{Comments on proximal extensions}\label{sec:prox_ext}
It is natural to try to extend geometric descent and optimal quadratic averaging to a proximal setting. For the sake of concreteness, let us focus on geometric descent. We can easily extend the suboptimal version of the algorithm to the proximal setting, but some difficulties arise when accelerating the method.
Suppose we are interested in solving the problem
$$\min_x\, f(x):=g(x)+h(x),$$
 where $g\colon\R^n\to\R$ is $\beta$-smooth and $\alpha$-strongly convex, and $h\colon\R^n\to\R\cup\{+\infty\}$ is closed, convex, and is such that the proximal mapping
  $$\text{prox}_{th}(x):=\argmin_z\,\{ h(z)+\frac{1}{2t}\|z-x\|^2\}$$ is easily computable.
In the analysis of first-order methods for such problems, the \emph{gradient mapping} $G_t(x) := \frac{1}{t} \left( x - \text{prox}_{th}( x - t \nabla g(x) ) \right)$ plays the role of the usual gradient. The following is a standard estimate; see for example \cite[Section 2.2.3]{nestBook}. We provide a proof for completeness.
 \begin{lem} \label{smoothcvxbnd}
Fix a step length $t>0$ and define a proximal gradient step $x^+ := x - t G_t(x)$.
Then for every $y\in \R^n$ the inequality holds:
\[ f(y) \geq f(x^+) + \ip{G_t(x)}{y - x} + t \left( 1 - \frac{\beta t}{2} \right) \tnorm{ G_t(x) }^2 + \frac{\alpha}{2} \tnorm{y - x}^2. \]
\end{lem}
\begin{proof}
Appealing to $\beta$-smoothness of $g$, we deduce
\[ f(x^+) \leq g(x) - t \ip{ \nabla g(x)}{ G_t(x) } +\frac{\beta t^2}{2} \tnorm{ G_t(x) }^2 + h(x^+). \]
Furthermore, strong convexity of $g$ implies
\[ f(x^+) \leq g(y) + \ip{\nabla g(x)}{x^+ - y} - \frac{\alpha}{2} \tnorm{y-x}^2 +\frac{\beta t^2}{2} \tnorm{ G_t(x) }^2 + h(x^+). \]
Finally, using the observation that $G_t(x) - \nabla g(x)$ belongs to $\partial h ( x^+ )$, we have
\[ f(x^+) \leq f(y) + \ip{G_t(x)}{ x^+ - y} - \frac{\alpha}{2} \tnorm{y-x}^2 +\frac{\beta t^2}{2} \tnorm{ G_t(x) }^2. \]
Rearrangement completes the proof.
\end{proof}

If we let $y = x^*$ in Lemma~\ref{smoothcvxbnd} and rearrange we get
\[ x^* \in \B{x- \frac{1}{\alpha} G_t(x)}{\left( \frac{1}{\alpha^2} - \frac{2}{\alpha} t + \frac{\beta}{\alpha} t^2 \right) \tnorm{G_t(x)}^2 - \frac{2}{\alpha} \left( f(x^+) - f^* \right)}. \]
How should we choose the step length $t$?
A simple approach is to choose $t$ to minimize the quantity $\frac{1}{\alpha^2} - \frac{2}{\alpha} t + \frac{\beta}{\alpha} t^2$, i.e., set $t = \frac{1}{\beta}$.
With this choice of $t$, we deduce the inclusion
\[ x^* \in \B{x^{++}}{\left( 1 - \frac{1}{\kappa} \right) \frac{\tnorm{G_{1/\beta}(x)}^2}{\alpha^2} - \frac{2}{\alpha} \left( f(x^+) - f^* \right) }, \]
where $x^{++} = x- \frac{1}{\alpha} G_{1/\beta}(x)$ is a \emph{long step} and $x^+ = x - \frac{1}{\beta} G_{1/\beta}(x)$ is a \emph{short step}.
A proximal version of the suboptimal geometric descent follows easily from Lemma~\ref{lem:ballIntersection1}.

To accelerate the proximal geometric descent algorithm we assume in iteration $k$ that $x^*$ lies in some ball
\[ \B{c_k}{R_k^2 - \frac{2}{\alpha} \left( f(y_k) - f^* \right) }. \]
We then consider a second minimizer enclosing ball derived from information at some point $x_{k+1}$:
\[ x^* \in \B{x_{k+1}^{++}}{\left( 1 - \frac{1}{\kappa} \right) \frac{\tnorm{G_{1/\beta}(x_{k+1})}^2}{\alpha^2} - \frac{2}{\alpha} \left( f(x_{k+1}^+) - f^* \right) }. \]
Following the same pattern as in Section~\ref{subsec:opt_geo}, if we choose $x_{k+1}$ to satisfy $f(x_{k+1}) \leq f(y_k)$ and appeal to the  smoothness inequality $f(x_{k+1}^+) \leq f(x_{k+1}) - \frac{1}{2\beta} \tnorm{G_{1/\beta}(x_{k+1})}^2$, we deduce the inclusion
\[ x^* \in \B{c_k}{R_k^2 -\frac{1}{\kappa} \frac{\tnorm{G_{1/\beta}(x_{k+1})}^2}{\alpha^2 }  -\frac{2}{\alpha} \left( f(x_{k+1}^+) - f^* \right) }. \]
By Lemma~\ref{lem:ballIntersection2} there is a new center $c_{k+1}$ with
\[ x^* \in \B{c_{k+1}}{ \left( 1 - \frac{1}{\sqrt{\kappa}} \right) R_k^2  -\frac{2}{\alpha} \left( f(x_{k+1}^+) -f^* \right) }, \]
provided the old centers $x_{k+1}^{++}$ and $c_k$ are far apart; specifically, we must be sure that the inequality
\[ \tnorm{x_{k+1}^{++} - c_k}^2 \geq \frac{\tnorm{G_{1/\beta}(x_{k+1})}^2}{\alpha^2} \qquad\textrm{holds}. \]
How do we choose $x_{k+1}$ to satisfy both $f(x_{k+1}) \leq f(y_k)$ and $\tnorm{x_{k+1}^{++} - c_k}^2 \geq \frac{\tnorm{G_{1/\beta}(x_{k+1})}^2}{\alpha^2}$?
The desired $x_{k+1}$ does exist; for example, $x_{k+1} = x^*$ is such a point.
In the proximal setting, it is not clear how to choose $x_{k+1}$ to ensure these two inequalities (even for specific problem classes).
This is an interesting topic for future research.

\subsection*{Acknowledgments}
We thank the anonymous referee for useful suggestions, which undoubtedly improved the quality of the paper. We also thank Stephen J. Wright for pointing out an important typo in the proof of Theorem~\ref{thm:optAvgConvergence} in an early version of the manuscript.

	\nopagebreak
	\bibliographystyle{plain}
	\bibliography{bibliography}

\appendix
\section{Exact line search in accelerated gradient descent}\label{sec:exact_line_search}

Nesterov's method is based on an \emph{estimate sequence}; that is, a sequence of functions $Q_k$ and nonnegative numbers $\Lambda_k$ with
\[ \Lambda_k \to 0 \quad \text{and} \quad Q_k(x) \leq (1 - \Lambda_k) f(x) + \Lambda_k Q_0(x). \]
Estimate sequences are useful because if $y_k$ satisfies $f(y_k) \leq v_k := \min_{x \in \R^n} Q_k(x)$, then
\[ f(y_k) - f^* \leq \Lambda_k \left( Q_0( x^* ) - f^* \right); \]
that is, $f(y_k)$ approaches $f^*$ with error proportional to $\Lambda_k$, see \cite{nestBook}.

The quadratics in Algorithm~\ref{alg:nestOpt} (with appropriately chosen $\Lambda_k$) form an estimate sequence.
To explain, for $k \geq 1$, pick vectors $x_{k}$ and numbers $\lambda_k \in (\delta,1)$ with $\delta > 0$.
Next, recursively define  
\begin{align*}
Q_0(x) &= v_0 + \frac{\gamma_0}{2} \tnorm{x - c_0}^2 \quad \text{and} \\
Q_{k}(x) &= (1-\lambda_{k}) Q_{k-1}(x) + \lambda_{k} \left(  f(x_{k}) - \frac{\tnorm{\nabla f(x_{k})}^2}{2 \alpha} + \frac{\alpha}{2} \tnorm{ x - x_{k}^{++}}^2 \right).
\end{align*}
Then the quadratics $Q_k$ and numbers $\Lambda_k = \prod_{j=1}^{k} (1 - \lambda_j)$ are an estimate sequence for $f$.
Nesterov's method is designed to ensure the inequality $f(x_k^+) \leq v_k$ with the added \emph{optimal rate condition} $\lambda_k \geq \sqrt{ \frac{\alpha}{\beta} }$.

The scheme in Algorithm~\ref{alg:nestOpt} with $x_{k}=\ls{c_{k-1}}{x_{k-1}^+}$ also guarantees these conditions.
Trivially we have $f(x_0^+) \leq v_0$.
Assume, for induction, that we have $f(x_{k-1}^+) \leq v_{k-1}$.
From \cite[Lemma 2.2.3]{nestBook}, we know
\begin{align*}
 v_{k} = (1 - \lambda_k) v_{k-1} + &\lambda_k f(x_{k}) - \frac{\lambda_k^2}{2 \gamma_{k}} \tnorm{ \nabla f (x_{k} ) }^2 + \\
 &+\frac{\lambda_k(1-\lambda_k) \gamma_{k-1}}{\gamma_{k}} \left( \frac{\alpha}{2} \tnorm{x_{k} - c_{k-1}}^2 + \ip{\nabla f(x_{k}) }{c_{k-1} - x_{k}} \right). 
\end{align*}
Since $x_{k} = \ls{c_{k-1}}{x_{k-1}^+}$, we have $f(x_{k}) \leq f(x_{k-1}^+) \leq v_{k-1}$ and $\ip{\nabla f(x_{k})}{c_{k-1} - x_{k}} = 0$, and therefore
\begin{align*}
v_{k} &\geq f(x_{k})  - \frac{\lambda_k^2}{2 \gamma_{k}} \tnorm{ \nabla f (x_{k} ) }^2 = f(x_{k}) - \frac{1}{2\beta} \tnorm{ \nabla f (x_{k} ) }^2 
\geq f(x_{k}^+).
\end{align*}
Provided we set $\gamma_0 \geq \alpha$, we get the optimal rate condition $\lambda_k = \sqrt{\frac{\gamma_{k}}{\beta}} \geq \sqrt{\frac{\alpha}{\beta}}$.

\end{document}